\documentclass[12pt,DIV=14]{scrartcl}
\usepackage{amsthm,amssymb,amsmath}
\usepackage{enumerate}
\usepackage{stmaryrd}
\usepackage
[colorlinks=true,citecolor=red,urlcolor=purple,linkcolor=blue,pdfborder={0 0 0}]{hyperref}

\makeatletter
\newsavebox{\@brx}
\newcommand{\lla}[1][]{\savebox{\@brx}{\(\m@th{#1\langle}\)}%
\mathopen{\copy\@brx\kern-0.5\wd\@brx\usebox{\@brx}}}
\newcommand{\rra}[1][]{\savebox{\@brx}{\(\m@th{#1\rangle}\)}%
\mathopen{\copy\@brx\kern-0.5\wd\@brx\usebox{\@brx}}}
\makeatother

\allowdisplaybreaks

\newcommand{\im}{{\operatorname{im}}}

\newcommand{\jump}[1]{{\llbracket #1 \rrbracket}}

\allowdisplaybreaks[1]
\theoremstyle{plain}
\newtheorem{theorem}{Theorem}
\newtheorem{lemma}[theorem]{Lemma}
\newtheorem{corollary}[theorem]{Corollary}

\theoremstyle{definition}
\newtheorem{remark}[theorem]{Remark}
\newtheorem{definition}[theorem]{Definition}

\newcommand{\ie}{i.\,e.}
\newcommand{\eg}{e.\,g.}
\newcommand{\spann}{\mathrm{span}\,}
\newcommand{\p}{\varphi}

\newcommand{\Om}{\Omega}
\newcommand{\pOm}{{\partial \Omega}}
\newcommand{\Th}{\mathcal{T}_h}
\newcommand{\Tnoh}{\mathcal{T}}

\newcommand{\T}{\mathcal{T}}

\newcommand*{\llbrace}{\ensuremath{\lbrace\hspace*{-0.235em}\vert}}
\newcommand*{\rrbrace}{\ensuremath{\vert\hspace*{-0.23em}\rbrace}}
\newcommand*{\avg}[1]{\ensuremath{\llbrace{#1}\rrbrace}}
\newcommand{\cupw}[1]{\ensuremath{\vecc_\mathrm{upw}(#1)}}

\newcommand{\R}{\mathbb{R}}

\newcommand{\N}{\mathbb{N}}
\newcommand{\eps}{\varepsilon}

\renewcommand{\nu}{{\mathfrak n}}

\renewcommand{\div}{\operatorname{div}}

\newcommand{\vecc}{\boldsymbol{c}}
\newcommand{\vecK}{\boldsymbol{K}}
\newcommand{\vecp}{\boldsymbol{p}}
\newcommand{\vecnu}{\boldsymbol{\nu}}
\newcommand{\vecxi}{\boldsymbol{\xi}}

\newcommand{\bzero}{\boldsymbol{0}}

\providecommand{\talpha}{\tilde{\alpha}}
\providecommand{\tg}{g}
\providecommand{\wtGamma}{\Gamma}

\title{Error estimates for completely discrete FEM in energy-type and weaker norms}
\author{Lutz Angermann\thanks{%
Dept.~of Mathematics, Clausthal University of Technology, Erzstr.~1, D-38678 Clausthal-Zellerfeld, Germany,
lutz.angermann@tu-clausthal.de},
Peter Knabner\thanks{%
Dept.~of Mathematics, University of Erlangen-Nuremberg, Cauerstr.~11, D-91058 Erlangen, Germany,
knabner@math.fau.de},
Andreas Rupp\thanks{%
School of Engineering Science, Lappeenranta-Lahti University of Technology, P.O.Box 20, FI-53851 Lappeenranta,
Finland,
Andreas.Rupp@lut.fi}}
\date{}

\begin{document}

\maketitle

\begin{abstract}
The paper presents error estimates within a unified abstract framework for the analysis of FEM
for boundary value problems with linear diffusion-convection-reaction equations
and boundary conditions of mixed type.
Since neither conformity nor consistency properties are assumed, the method
is called completely discrete. We investigate two different stabilized discretizations
and obtain stability and optimal error estimates in energy-type norms and,
by generalizing the Aubin-Nitsche technique, optimal error estimates in weaker norms.
\end{abstract}



\section{Introduction}

In this paper we present a unified approach to the analysis of FEM
for boundary value problems with linear elliptic differential equations
of the second order, where, in addition to the diffusion-convection-reaction structure
of the partial differential equation, mixed type boundary conditions
(first, second, third kind) are taken into account.
We allow completely discrete formulations in the sense that the discrete FE spaces
are not necessarily embedded into the spaces of the weak formulation of the boundary
value problem and -- based on a suitable notion of consistency -- that the
discrete problems do not have to be consistent.
In addition to stability and error estimates in energy-type norms,
a generalization of the Aubin-Nitsche technique
for obtaining error estimates in weaker norms is discussed.

Selected aspects of our exposition are not entirely new or even not really profound;
this applies, for example, to our version of Strang's second lemma \cite{Strang:72b},
which is circulating in the literature in several, slightly different versions.
There are also other variants of the generalization of the duality argument;
we refer here -- also for an overview -- to \cite{DiPietro:18}.
The abstract results obtained are applied to two concrete discretizations
-- a Crouzeix-Raviart discretization of order one and more general
discontinuous Galerkin methods of the IPG type.
For both cases we discuss the stability, consistency and convergence properties
that result from the general theory.

As for the theoretical aspects, the work already mentioned \cite{DiPietro:18}
and its extension \cite{DiPietro:21} are perhaps the most closely related publications
to our work.
Compared to these, we prefer not to include an  interpolation (or quasi-interpolation) operator
in the consistency definition,
but rather require that the discrete bilinear form can be extended in such a manner
that the solution of the continuous problem
(which often has more regularity than the elements of spaces in which
the weak formulation of the boundary value problem is given), belongs to the extended domain of
definition (as in \cite[Def.~1.31]{DiPietro:12}).
Furthermore, the extended paper \cite{DiPietro:21} applies its theory
only to a pure diffusion problem under homogeneous Dirichlet boundary conditions.

The paper is structured as follows.
In the subsequent section we present, in an abstract framework, error estimates
for the solution of discretized variational equations in both energy-type norms and weaker norms,
whereby neither conformity nor consistency are assumed.
Then, in Section \ref{sec:modelproblem}, we describe the model problem, the solution of which is to be approximated,
and the most important prerequisites. The model problem is a boundary value problem for
a scalar diffusion-convection-reaction equation with boundary conditions of mixed type.
Sections \ref{sec:NonconformingFE} and \ref{sec:DGM} describe the specific discretizations,
including their stabilization mechanisms,
and the theoretical results from Section \ref{sec:GM_VarEq} are applied.
In both situations and under reasonable assumptions, optimal error estimates are obtained.

\section{General variational equations}
\label{sec:GM_VarEq}
Given two real Banach spaces $(U,\|\cdot\|_U)$, $(V,\|\cdot\|_V)$,
this paper is concerned with the application of finite element methods
to approximate the solution to the following problem:
\begin{equation}\label{EQ:VariationalFormulationGeneral}
\text{Find } u \in U \text{ such that} \quad a(u,v) = \ell(v)
\quad\text{for all } v \in V,
\end{equation}
where throughout the paper $\ell:\; V \to \R$ is a continuous linear form, and
$a:\; U \times V \to \R$ is a continuous bilinear form.

In the above setting, the following result about the existence of a unique solution is known.

\begin{theorem}\label{th:BanachNecasBabuska}
Let $V$ be reflexive.
Then the problem
\eqref{EQ:VariationalFormulationGeneral} is uniquely solvable
for every right-hand side $\ell\in V'$
if and only if the following two conditions are satisfied:
\begin{align}
& \alpha := \inf_{{u \in U\setminus\{0\}}} \sup_{{v \in V\setminus\{0\}}}
\frac{a(u,v)}{\|u\|_{U} \|v\|_{V}} > 0,
\label{eq:vfg1}\\
& a(u,v) = 0 \quad\text{for all } u \in U \quad \Longrightarrow \quad v = 0
\quad\text{for } v \in V.
\label{eq:vfg2}
\end{align}
If both conditions are met,
the solution $u\in U$ of \eqref{EQ:VariationalFormulationGeneral}
satisfies the stability estimate
\[
\|u\|_{U} \le \frac{1}{\alpha} \|\ell\|_{V'}.
\]
\end{theorem}
\begin{proof}
First we note that the variational equation \eqref{EQ:VariationalFormulationGeneral}
can be reformulated as an operator equation:
\begin{equation*}
A u = \ell,
\end{equation*}
where $A:\; U \to V'$ is defined by means of the relationship
$(Au)(v) := a(u,v)$ for all $u\in U$, $v\in V$.
The continuity of $a$ implies the continuity of $A$.

The assertion follows from the following chain of arguments,
but we omit their detailed proofs.
\begin{enumerate}[1)]
\item
Let $U,V$ be normed spaces only. Then:
\[
\eqref{eq:vfg1} \quad \Longleftrightarrow \quad A^{-1}:\;\im{(A)}\to U
\text{ exists and is continuous}.
\]
\item
If, in addition to the assumption in 1), $U$ is complete, \ie\ a Banach space,
then $\overline{\im{(A)}} = \im{(A)}$.
\item
Let, in addition to the assumptions in 2), $V$ be a reflexive Banach space.
Then:
\[
\eqref{eq:vfg2} \quad \Longleftrightarrow \quad \im{(A)}=V'.
\]
\end{enumerate}
\end{proof}
To discretize the problem \eqref{EQ:VariationalFormulationGeneral} formally we consider
real Banach spaces $(U_h,\|\cdot\|_{U_h})$,
\linebreak
$(V_h,\|\cdot\|_{V_h})$
of the same finite dimension,
a bilinear form $a_h:\; U_h \times V_h \to \R$ and
a linear form $\ell_h:\; V_h \to \R$.
Here the index $h$ is a positive parameter
(typically an element of a sequence of positive real numbers with accumulation point 0)
such that the dimension of $U_h$ and $V_h$ increases unbounded as $h$ approaches zero.
The corresponding discrete problem reads as follows:
\begin{equation}\label{eq2018:6.12}
\text{Find } u_h \in U_h \text{ such that } \quad
a_h(u_h,v_h) = \ell_h(v_h) \quad\text{for all } v_h \in V_h.
\end{equation}
We call the discretization \eqref{eq2018:6.12} \emph{conforming},
if $U_h \subset U$ as well as $V_h \subset V$, otherwise \emph{nonconforming}.
In the conforming case, we set
$\|\cdot\|_{U_h}:=\|\cdot\|_U$ and $\|\cdot\|_{V_h}:=\|\cdot\|_V$
unless differently specified.

In the analysis of the nonconforming case, the \emph{augmented spaces}
\[
U(h) := U + U_h
\quad\text{and}\quad
V(h) := V + V_h
\]
will be useful.
This implicitly assumes the existence of linear superspaces
for $U,U_h$ and $V,V_h$, respectively.
Furthermore we assume that the spaces $U(h)$, $V(h)$ are equipped
with norms $\|\cdot\|_{U(h)}$, $\|\cdot\|_{V(h)}$, respectively.
It is often desirable to take advantage of additional knowledge about the solution
of the problem \eqref{EQ:VariationalFormulationGeneral}, \eg\ certain regularity properties.
In such a case it is natural to introduce a proper subspace, say $W\subset U$,
as the solution space, which may have a stronger topology.

\begin{definition}\label{def:consistentapprox}
Let $u\in W\subset U$ be the solution of the problem \eqref{EQ:VariationalFormulationGeneral}.
The discrete formulation \eqref{eq2018:6.12} is called \emph{consistent}
on the subspace $W\subset U$,
if the discrete bilinear form $a_h$ can be extended onto the product space $(W+U_h) \times V_h$
(keeping the notation $a_h$) such that
\[
a_h(u,v_h) = \ell_h(v_h) \quad\text{for all } v_h \in V_h
\]
holds.
Otherwise the discrete formulation \eqref{eq2018:6.12} is called \emph{nonconsistent}.
\end{definition}
If the discrete formulation \eqref{eq2018:6.12} is consistent,
a solution $u_h \in U_h$ has the following property of \emph{Galerkin orthogonality}:
\begin{equation}\label{eq:Galorthogonality}
a_h(u-u_h, v_h) = 0 \quad\text{for all } v_h \in V_h.
\end{equation}
This property is lost in the nonconsistent situation,
since additional terms appear on the right-hand side, which are generally nontrivial.

An extension of the standard convergence analysis for the consistent conforming case
is given by the following generalization of Strang's second lemma.
Extensions of this kind can be found in the literature
(\eg\ \cite{DiPietro:12}, \cite{Cangiani:17})
but in fact all of these results (including our subsequent Theorem~\ref{th:Strang2})
are not very deep and rather technical in nature, but allow the analysis
of more general finite element (and related) methods.
\begin{theorem}\label{th:Strang2}
Let $u\in W\subset U$ be the solution of the problem \eqref{EQ:VariationalFormulationGeneral}.
Assume that the norm $\|\cdot\|_{U_h}$ can be extended to a norm on $W+U_h$
(keeping the notation $\|\cdot\|_{U_h}$),
and the condition
\begin{equation}\label{EQ:NB2h}
\alpha_h :=
\inf_{u_h \in U_h\setminus\{0\}}
\sup_{v_h \in V_h\setminus\{0\}}
\frac{a_h(u_h,v_h)}{\|u_h\|_{U_h}\|v_h\|_{V_h}} > 0
\end{equation}
is satisfied.

Further assume that the discrete bilinear form $a_h$ can be continuously extended onto
the product space $\left(W+U_h,\|\cdot\|_{U(h)} \right) \times \left( V_h, \|\cdot\|_{V_h} \right)$
(keeping the notation $a_h$),
\ie, there exists a constant $\widetilde{M}_h \ge 0$ such that
\begin{equation}\label{eq:contah3}
|a_h(w, v_h)| \le \widetilde{M}_h \| w \|_{U(h)} \| v_h \|_{V_h}
\quad\text{for all } w \in W+U_h,\ v_h \in V_h.
\end{equation}
Then the estimate
\begin{align*}
\| u-u_h \|_{U_h}
&\le\inf_{w_h \in U_h}
\left(\frac{\widetilde{M}_h }{\alpha_h}\| u-w_h \|_{U(h)} + \| u-w_h \|_{U_h} \right)
+ \frac{1}{\alpha_h}\sup_{v_h \in V_h\setminus\{0\}}
\frac{| a_h(u, v_h) - \ell_h(v_h) |}{\| v_h \|_{V_h}}
\end{align*}
holds.
\end{theorem}

\begin{proof}
From \eqref{EQ:NB2h}, \eqref{eq:contah3} and
\begin{align*}
a_h(u_h - w_h, v_h)
=\ell_h(v_h) - a_h(u,v_h) + a_h(u-w_h,v_h)
\quad\text{for any }w_h\in U_h
\end{align*}
it follows immediately that
\begin{align*}
\alpha_h \| u_h - w_h \|_{U_h}
\le \sup_{v_h \in V_h\setminus\{0\}}
\frac{| a_h(u,v_h) - \ell_h(v_h) |}{\| v_h \|_{V_h}}
+ \widetilde{M}_h \| u - w_h \|_{U(h)}.
\end{align*}
The triangle inequality
$\|u-u_h\|_{U_h} \le \|u-w_h\|_{U_h} + \|w_h-u_h\|_{U_h}$
yields the result.
\end{proof}

\begin{remark}\phantomsection\label{folg:DG_analysis_conv}
\begin{enumerate}[1)]
\item
Let the assumptions of Theorem~\ref{th:Strang2} be satisfied.
If there is a constant $\widetilde{C}_h>0$ such that
\[
\| w_h \|_{U(h)} \le \widetilde{C}_h \| w_h \|_{U_h}
\quad\text{for all } w_h \in U_h,
\]
we can also apply the triangle inequality w.r.t.\ the $U(h)$-norm and conclude
\[\| u-u_h \|_{U(h)}
\le \left(1+ \frac{ \widetilde{C}_h \widetilde{M}_h }{\alpha_h}\right)
\inf_{w_h \in U_h}\| u-w_h \|_{U(h)}
+ \frac{\widetilde{C}_h}{\alpha_h}\sup_{v_h \in V_h\setminus\{0\}}
\frac{| a_h(u, v_h) - \ell_h(v_h) |}{\| v_h \|_{V_h}}\,.
\]
\item
Let the assumptions of Theorem~\ref{th:Strang2} be satisfied.
If there is a constant $\widetilde{C}_h>0$ such that
\[
\| w_h \|_{U_h} \le \widetilde{C}_h \| w_h \|_{U(h)}
\quad\text{for all } w \in W + U_h,
\]
then the estimate
\[
\|u - u_h\|_{U_h} \le \left(\widetilde{C}_h + \frac{\widetilde{M}_h}{\alpha_h} \right)
\inf_{w_h \in U_h} \| u - w_h \|_{U(h)}
+ \frac{1}{\alpha_h}\sup_{v_h \in V_h\setminus\{0\}}
\frac{| a_h(u, v_h) - \ell_h(v_h) |}{\| v_h \|_{V_h}}
\]
holds.
\end{enumerate}
\end{remark}

\subsubsection*{Error estimates in weaker norms}
In the standard finite element literature for second-order linear elliptic
boundary value problems, the so-called
Aubin-Nitsche duality argument \cite{Aubin:67}, \cite{Nitsche:68}
is used to establish error estimates in norms which are weaker than
the natural energy norm (or equivalent norms), typically in the $L^2$-norm.
The main ingredient is an auxiliary variational problem of the form

\medskip\quad
Find $v \in V$ such that
\begin{equation}\label{eq:adjoint}
a(w,v) = \tilde\ell(w) \quad\text{for all } w \in U,
\end{equation}
where $\tilde\ell \in U'$ is a suitably chosen continuous linear form.

Here we extend this setting to a more general framework,
in which the occurring discrete spaces no longer have to be subspaces
of the ``continuous'' spaces $U,V$ of the weak formulation \eqref{EQ:VariationalFormulationGeneral},
and the discretization does not necessarily have to be consistent.
This abstract framework is applied to two examples of nonconforming FEM for
diffusion-convection-reaction equations in Sections \ref{sec:NonconformingFE} and \ref{sec:DGM}.
However, the range of models and numerical techniques covered by the analytical framework
goes well beyond these examples.

We basically assume here that there exists a reflexive Banach space $Z$ such that
$U+U_h \subset Z'$ and denote by $\lla \cdot, \cdot \rra:\;Z'\times Z''\to\R$
the duality pairing on $Z'\times Z''$.
Thanks to the reflexivity of $Z$ we may identify the bidual space $Z''$ with $Z$: $Z'' \cong Z$.

Next we specify the right-hand side of the adjoint variational problem \eqref{eq:adjoint} as
\begin{equation*}
\tilde\ell(w) := \tilde\ell_{g}(w) := \lla w , \, g \rra
\quad\text{for all } 
w \in U,
\end{equation*}
where $g \in Z$ is arbitrary.
That is, the particular problem

\medskip\quad
Find $v \in V$ such that
\begin{equation}\label{eq:specialadjoint}
a(w,v) = \lla w , \, g \rra \quad\text{for all } w \in U
\end{equation}
is considered below.
Regarding the solvability of the problem \eqref{eq:specialadjoint}
we assume that there exists not only a unique solution $v_{g}\in V$, but that
it belongs to some proper subspace $Y\subset V$ (similar to the original (``primal'') problem).
We further assume that the discrete variational formulation

\medskip\quad
Find $v_{gh} \in V_h$ such that
\begin{equation}\label{eq:discrspecialadjoint}
a_h (w_h, v_{gh}) = \lla w_h, g \rra
\quad\text{for all }
w_h \in U_h
\end{equation}
possesses a unique solution.

The following result, on which perhaps the most interesting
part of the analysis of the concrete examples in
Sections 4 and 5 is based, has already been published in the book \cite[Lemma~6.11]{Angermann:21a}.
Nevertheless, we present the proof here in a revised, shortened version,
since it is referred to in the discussion that follows the proof
and especially in the analysis of the two examples.
\begin{theorem}\label{l:weakererror}
Let $u\in W\subset U$ be the solution of the problem \eqref{EQ:VariationalFormulationGeneral},
$u_h \in U_h$ the solution of the discrete problem \eqref{eq2018:6.12},
$v_{g}\in Y\subset V$ the solution of the adjoint problem \eqref{eq:specialadjoint},
and $v_{gh} \in V_h$ the solution of the discrete adjoint problem \eqref{eq:discrspecialadjoint}.
Assume that the bilinear form $a_h$ can be continuously extended onto the product space $(W+U_h) \times (Y+V_h)$
(keeping the notation $a_h$), \ie, there exists a constant $\widetilde{M}_h \ge 0$ such that the estimate
\begin{equation}\label{eq:contofah}
|a_h (w,z)| \le \widetilde{M}_h \| w \|_{U(h)} \| z \|_{V(h)}
\quad\text{for all }
w\in W+U_h,\ z\in Y+V_h
\end{equation}
holds.
Finally assume that the linear form $\ell_h$ can be extended onto $Y+V_h$
(keeping the notation $\ell_h$, too).
Then:
\begin{equation}\label{eq2020:kap61_6174}
\begin{aligned}
\|u - u_h\|_{Z'}
& \le \sup_{g \in Z \setminus\{0\}}
\frac{1}{\|g \|_{Z}} \Big\{ \widetilde{M}_h \| u - u_h\|_{U(h)} \| v_{g} - v_{gh} \|_{V(h)}
- \left[ a_h (u-u_h, v_{g}) - \lla u - u_h, g \rra \right] \\
&\quad - \left[ a_h (u,v_{g} - v_{gh}) - \ell_h(v_{g} - v_{gh}) \right ]
+ \left[ (a_h - a) (u,v_{g}) - (\ell_h - \ell) (v_{g})\right] \Big\}.
\end{aligned}
\end{equation}
\end{theorem}
\begin{proof}
Thanks to the relationship
\[
\|u - u_h\|_{Z'}
= \sup_{g \in Z \setminus\{0\}} \frac{ \lla u - u_h,g\rra}{\|g \|_{Z}}
\]
it is sufficient to estimate the numerator term. It can be decomposed as
\begin{align*}
&\lla u - u_h,g\rra\\
& = a(u,v_{g}) - a_h(u_h, v_{gh}) \\
& = a_h (u,v_{g}) - a_h(u_h,v_{gh}) + (a-a_h)(u,v_{g}) \\
& = a_h (u - u_h, v_{g} - v_{gh}) + a_h (u_h, v_{g} - v_{gh}) + a_h (u -u_h, v_{gh}) + (a -a_h) (u,v_{g})\\
& =: I_1 + I_2 + I_3 + I_4.
\end{align*}
For the first term we have the estimate \eqref{eq:contofah} by assumption:
\[
|I_1| \le \widetilde{M}_h \| u -u_h \|_{U(h)} \| v_{g} - v_{gh}\|_{V(h)}.
\]
The term $I_3$ can be split as follows:
\begin{align*}
I_3 & = -a_h (u,v_{g} - v_{gh}) + a_h (u,v_{g}) - a_h (u_h, v_{gh}) \\
& = - [ a_h (u, v_{g} - v_{gh}) - a (u,v_{g}) + \ell_h(v_{gh}) ] + (a_h - a) (u, v_{g}) \\
& = - [ a_h (u, v_{g} - v_{gh}) - \ell_h(v_{g} - v_{gh}) ] + (a_h - a)(u,v_{g}) - (\ell_h - \ell) (v_{g}).
\end{align*}
Finally it is not difficult to see that
\begin{align*}
I_2 + I_4 & = -a_h (u - u_h, v_{g}) - a_h (u_h, v_{gh}) + a(u,v_{g})\\
& = -a_h (u - u_h, v_{g}) - \lla u_h,g \rra + \lla u,g \rra\\
& = - [ a_h ( u - u_h, v_{g}) - \lla u - u_h,g \rra ].
\end{align*}
Putting all the above relationships together, we obtain the statement.
\end{proof}
The properties of the variational formulations for different situations are
summarized in the subsequent table.
It should be read so that the relationships in the second or third column apply
in addition to the relationships listed in the first column.
\begin{center}
\begin{tabular}{|c|c|c|}\hline
General case & Conforming case & Consistent case\\\hline
$a(u,v)=\ell(v)$ & & \\
& $a(u,v_h)=\ell(v_h)$ &\\
$a_h(u_h,v_h)=\ell_h(v_h)$ & & \\
& & $a_h(u,v_h)=\ell_h(v_h)$\\
$a(w,v_g)=\lla w,g \rra$ & & \\
& $a(w_h,v_g)=\lla w_h,g \rra$ & \\
$a_h(w_h,v_{gh})=\lla w_h,g \rra$ & & \\
& & $a_h(w_h,v_g)=\lla w_h,g \rra$
\\\hline
\end{tabular}
\end{center}
If both the original and the adjoint problem are discretized by
conforming methods (see the second column of the table),
the term $I_2 + I_3 + I_4$
(\ie, the last three terms in \eqref{eq2020:kap61_6174})
can be rewritten as
\begin{equation}\label{eq2020:kap61_6180}
I_2 + I_3 + I_4
= (a-a_h)(u - u_h, v_{g} - v_{gh}) - (a - a_h) (u_h, v_{gh}) + (\ell - \ell_h) (v_{gh}) .
\end{equation}
If the discretizations of both the original and the adjoint problem
are consistent (see the third column of the table), we have that
\begin{equation}\label{eq2020:kap61_61781}
I_2 + I_3 + I_4
= (a - a_{h}) (u,v_{g}).
\end{equation}

\subsubsection*{Discussion}
The success of the presented approach clearly depends on the possibility of obtaining
suitable estimates of the four individual terms in the bound in \eqref{eq2020:kap61_6174}.
In particular, all addends standing in the braces have to be estimated in such a way
that they contain the term $\|g\|_{Z}$ as a factor.

Concerning the first term, we can write, for $g \in Z \setminus\{0\}$,
\begin{align*}
\widetilde{M}_h \| u -u_h \|_{U(h)} \| v_{g} - v_{gh}\|_{V(h)}
&= \widetilde{M}_h \| u -u_h \|_{U(h)} \frac{\| v_{g} - v_{gh}\|_{V(h)}}{\|g\|_{Z}}\|g\|_{Z}\\
&\le \widetilde{M}_h \tilde\eta \| u -u_h \|_{U(h)} \|g\|_{Z}
\end{align*}
with
\[
\tilde\eta
:= \tilde\eta(A^\prime,A_h^\prime,V_h,Z)
:= \sup_{g \in Z \setminus\{0\}}\frac{\| (A^\prime)^{-1}g - (A_h^\prime)^{-1}g\|_{V(h)}}{\|g\|_{Z}},
\]
where
$(A^\prime)^{-1}:\;Z \to Y$ and $(A_h^\prime)^{-1}:\;Z \to V_h$ are the solution operators
of the problems \eqref{eq:specialadjoint} and \eqref{eq:discrspecialadjoint}, respectively.

The quantity $\tilde\eta$ can be usefully further estimated if, for example,
$(Y,\| \cdot \|_{Y})$ is a Banach space (in fact it was sufficient if $\| \cdot \|_{Y}$ were a seminorm)
and the following two conditions are met:
\begin{itemize}
\item
Stable regularity of the solution of the adjoint problem:
The solution $v_{g}$ of \eqref{eq:specialadjoint}
even belongs to the space $Y$ and satisfies the stability estimate
\begin{equation}\label{eq2020:kap61_6176}
\| v_{g} \|_{Y} \le C_{s} \|g\|_{Z}
\quad\text{for all }
g \in Z,
\end{equation}
where $C_{s} \ge 0$ is a constant independent of $g$.
\item
Convergence of the solution of the discrete adjoint problem:
There exist constants $C_{a} \ge 0$, $q > 0$ such that the error of
the discrete solution $v_{gh} \in V_h$ of \eqref{eq:discrspecialadjoint}
satisfies the estimate
\begin{equation}\label{eq2020:kap61_6175}
\| v_{g} - v_{gh} \|_{V(h)} \le C_{a} h^q \| v_{g} \|_{Y}.
\end{equation}
\end{itemize}
Indeed, if both conditions are satisfied, we have that
\[
\| (A^\prime)^{-1}g - (A_h^\prime)^{-1}g\|_{V(h)} = \| v_{g} - v_{gh} \|_{V(h)} \le C_{a} h^q \| v_{g} \|_{Y}
\le C_{s}C_{a} h^q \|g\|_{Z},
\]
hence
\[
\tilde\eta \le C_{s}C_{a} h^q.
\]
So if the order of convergence of the discrete solution $u_h\in U_h$
of the original discrete problem \eqref{eq2018:6.12} is $p>0$, that is
\begin{equation}\label{eq:primalconvergence}
\| u - u_h \|_{U(h)} \le C_{p} h^p \| u \|_{W}
\end{equation}
with some constant $C_{p} \ge 0$, then we get the estimate
\begin{equation}\label{eq:wnbound1}
\widetilde{M}_h \| u -u_h \|_{U(h)} \| v_{g} - v_{gh}\|_{V(h)}
\le \widetilde{M}_h C_{s}C_{a}C_{p} h^{p+q} \| u \|_{W} \|g\|_{Z}.
\end{equation}
The second addend in the bound in \eqref{eq2020:kap61_6174} can be interpreted as a consistency
error of the discrete adjoint problem at the test function $u-u_h$.
If it were possible to obtain a consistency error estimate of the form
\begin{equation}\label{eq2020:kap61_6177}
| a_h (w,v_{g}) - \lla w,g \rra |
\le C_{ca} h^{\alpha} \| w \|_{U(h)} \| v_{g} \|_{Y},
\end{equation}
where $C_{ca} \ge 0$, $\alpha > 0$ are certain constants, this,
together with the regularity condition \eqref{eq2020:kap61_6176}
and the estimate \eqref{eq:primalconvergence}, would lead to the relationship
\begin{equation}\label{eq:wnbound2}
| a_h (u - u_h,v_{g}) - \lla u - u_h,g \rra |
\le C_{ca} C_{s} h^{\alpha} \|u - u_h\|_{U(h)}\|g\|_{Z}
\le C_{ca} C_{p} C_{s} h^{p + \alpha} \| u \|_{W} \|g\|_{Z}.
\end{equation}
The third addend in the bound in \eqref{eq2020:kap61_6174} is a consistency
error of the discrete original problem at the test function $v_{g} - v_{gh}$.
If an estimate of the type
\[
| a_h (u,z) - \ell_h(z) | \le C_{cp} h^{\beta} \| z \|_{V(h)} \| u \|_W
\]
with certain constants $C_{cp} \ge 0$, $\beta > 0$ is assumed, then,
similar to the above discussion of the first addend, it follows that
\begin{equation}\label{eq:wnbound3}
| a_h (u,v_{g} - v_{gh}) - \ell_h(v_{g} - v_{gh}) |
\le C_{cp} h^{\beta} \tilde\eta \| u \|_W \|g\|_{Z}
\le C_{s} C_{a} C_{cp} h^{q+\beta} \| u \|_W \|g\|_{Z}.
\end{equation}
The fourth addend represents approximation errors.
If it were possible to obtain an estimate of the form
\[
| (a_h -a) (u,v_{g}) - (\ell_h - \ell) (v_{g}) |
\le C_{q} h^{\gamma}  \| u \|_{W} \| v_{g} \|_Y
\]
with constants $C_{q} \ge 0$, $\gamma > 0$,
then it would follow, together with the regularity condition \eqref{eq2020:kap61_6176},
that
\begin{equation}\label{eq:wnbound4}
| (a_h -a) (u,v_{g}) - (\ell_h - \ell) (v_{g}) |
\le C_{q} C_{s} h^{\gamma} \| u \|_{W} \|g\|_{Z}.
\end{equation}
Putting the estimates \eqref{eq:wnbound1}--\eqref{eq:wnbound4} together,
the estimate \eqref{eq2020:kap61_6174} reads as
\[
\|u - u_h\|_{Z'}
\le C h^r \| u \|_{W}
\]
with
\[
C:= \widetilde{M}_h C_{s}C_{a}C_{p}
+ C_{ca} C_{p} C_{s}
+ C_{s} C_{a} C_{cp}
+ C_{q} C_{s}
\quad\text{and}\quad
r:=\min \{ p+q, \, p+\alpha, \, q+\beta, \, \gamma \} .
\]
In the frequently encountered case  $U = V \subset H^1(\Om)$ and $U_h = V_h$ consisting
of conforming $\mathcal{P}_{k}$-elements, a natural choice for the space $Z$ is
\[
Z := L^2(\Om).
\]
Provided the data of the boundary value problem are sufficiently smooth,
the relationships \eqref{eq2020:kap61_6176}, \eqref{eq2020:kap61_6175} can be satisfied
by choosing $Y := H^2(\Om)$ and $q=1$.

In other, less standard situations, $Z := H^{k-1}(\Om)$ can also be taken.
This may allow to choose $Y := H^{k+1}(\Om)$ and also $q = k$.
Then, provided that the consistency errors behave appropriately,
the the optimal case of order doubling in the norm $\| \cdot \|_{1-k}$ can be reached.

The importance of such negative norm estimates consists in the possibility
of deriving error estimates for some functionals $\p \in U(h)'$.
Indeed, assume that such a functional $\p$ is represented via a smooth function,
\ie, it even holds that
\[
\p \in Z \quad (\cong Z'').
\]
Then
\[
\p (u - u_h) \le \| \p \|_{Z} \|u - u_h\|_{Z'}\,.
\]
\begin{remark}
\begin{enumerate}[1)]
\item
If the discrete formulation \eqref{eq2018:6.12} is consistent, the first addend
in the bound in \eqref{eq2020:kap61_6174} can be replaced by
$\widetilde{M}_h \| u -u_h \|_{U(h)} \| v_{g} - v_{h}\|_{V(h)}$
with arbitrary $v_{h}\in V_h$
thanks to the Galerkin orthogonality \eqref{eq:Galorthogonality}.
Then we have the estimate
\[
\widetilde{M}_h \| u -u_h \|_{U(h)} \| v_{g} - v_{h}\|_{V(h)}
\le \widetilde{M}_h \eta \| u -u_h \|_{U(h)} \|g\|_{Z}
\]
with
\[
\eta := \eta(A^\prime,V_h,Z)
:= \inf_{v_h \in V_h} \sup_{g \in Z \setminus\{0\}}\frac{\| (A^\prime)^{-1}g - v_h\|_{V(h)}}{\|g\|_{Z}}
\quad (\le \tilde\eta).
\]
The quantity $\eta$ was introduced in \cite{Sauter:06} in connection with
stability and convergence investigations of conforming and consistent Galerkin discretizations
of the Helmholtz equation for large wavenumbers.
\item
In case of conforming discretizations and if $U = V$ are Hilbert spaces,
an alternative estimate can be derived.
Under the assumptions of Theorem~\ref{l:weakererror} we have
(using the notation of the proof):
\begin{align*}
\lla u - u_h,g \rra & = a(u - u_h, v_{g})
= a(u - u_h, v_{g} - \Pi_h v_{g}) + \ell(\Pi_h v_{g}) - a(u_h, \Pi_h v_{g}),
\end{align*}
where $\Pi_h:\;V_h\to V$ is the orthogonal projector.
Then
\begin{align*}
\|u - u_h\|_{Z'}
\le \sup_{g \in Z \setminus \{0\}} \frac{1}{\|g\|_{Z}} \Big\{ \widetilde{M} \|u-u_h\|_{V}
\|v_{g} - \Pi_h v_{g}\|_{V}
- [a (u_h, \Pi_h v_{g}) - \ell(\Pi_h v_{g})]\Big\}.
\end{align*}
Since $V$ is reflexive as a Hilbert space, we can set $Z: = V$.
Introducing $\ell_h := a_h(u_h, \cdot)$,
the second term can be treated as follows:
\begin{align*}
a (u_h, \Pi_h v_{g}) - \ell(\Pi_h v_{g})
&= (a-a_h) (u_h, \Pi_h v_{g}) + \ell_h(\Pi_h v_{g}) - \ell(\Pi_h v_{g}).
\end{align*}
Thanks to the the symmetry of $\Pi_h$ it holds that
\[
\ell_h(\Pi_h v_{g}) - \ell(\Pi_h v_{g})
=\lla \ell_h - \ell, \Pi_h v_{g} \rra
=\lla \ell_h - \Pi_h \ell, v_{g} \rra.
\]
Hence we arrive at the estimate
\[
|a (u_h, \Pi_h v_{g}) - \ell(\Pi_h v_{g})|
\le (a-a_h) (u_h, \Pi_h v_{g})
+ \|v_{g}\|_{V} \|\ell_h - \Pi_h \ell\|_{V}.
\]
\end{enumerate}
\end{remark}

\section{The model problem}
\label{sec:modelproblem}
In this and the subsequent section we will apply the theoretical results to
finite element discretizations of the following diffusion-convection-reaction problem.
Given a bounded polyhedral Lipschitz domain $\Om\subset\R^d$, $d\in\{2,3\}$,
we consider the differential equation in divergence form
\begin{equation}\label{eq:-1.33lin}
- \nabla\cdot(\vecK\nabla u - \vecc\,u) + r\,u = f
\quad\text{in }\Om
\end{equation}
with the data
\begin{gather*}
\vecK=\vecK^\top\in L^\infty(\Om)^{d,d},
\quad
\vecc\in L^\infty(\Om)^d,
\quad
\partial_j c_j \in L^{3/2}(\Om),\ j\in\{1,\ldots,d\},
\\
r\in L^\infty(\Om),
\quad
f\in L^2(\Om),
\end{gather*}
where, for some constant $k_0>0$,
\begin{equation}\label{eq:Kspd}
\vecxi\cdot(\vecK(x)\vecxi) \ge k_0 |\vecxi|^2
\quad\text{for all }\vecxi\in\R^d
\text{ and almost all } x\in\Om.
\end{equation}
To formulate the boundary conditions we assume that the boundary $\pOm$
is decomposed into disjoint subsets
$\wtGamma_j,$ $j\in\{1,2,3\}$:
$
\pOm = \wtGamma_1 \cup \wtGamma_2 \cup \wtGamma_3 \, ,
$
where $\wtGamma_3$ is supposed to be a relatively closed subset of $\pOm$.
Given the functions
$\tg_j\in L^2(\wtGamma_j)$, $j\in\{1,2\}$, and
$\talpha\in L^\infty(\wtGamma_2)$,
the boundary conditions are
(the symbol $\nu$ denotes the outer unit normal on $\pOm$):
\begin{equation}\label{eq2020:kap_32_334_ofdivergence}
\begin{alignedat}{2}
( \vecK \nabla u - \vecc\,u ) \cdot \vecnu
&= \tg_1 &\quad&\text{on } \wtGamma_1, \\
( \vecK \nabla u - \vecc\,u ) \cdot \vecnu + \talpha u
&= \tg_2 & &\text{on } \wtGamma_2,\\
u &= 0
& &\text{on } \wtGamma_3,
\end{alignedat}
\end{equation}
\ie, we restrict our investigations to homogeneous Dirichlet boundary conditions.

The variational formulation \eqref{EQ:VariationalFormulationGeneral} is specified
as follows:
\[
\begin{aligned}
U := V &:= \big\{ v \in H^1(\Om)\;|\;
v = 0 \text{ on } \wtGamma_3 \text{ in the sense of traces}\big\},\\
a(u,v) &:= (\vecK\nabla u - \vecc u,\nabla v) + (ru,v) + (\talpha u,v)_{\wtGamma_2}
\quad\text{for all } u,v \in V,\\
\ell(v) &:= (f,v) + (\tg_1,v)_{\wtGamma_1}
+ (\tg_2,v)_{\wtGamma_2}
\quad\text{for all } v \in V.
\end{aligned}
\]
Since we will also have to deal with the adjoint boundary value problem,
we formulate the additional requirements to the data in a form that slightly differs
from the usual one:
\begin{equation}\label{eq:coercivityconddviform}
\begin{aligned}
\text{1) } &
r + \frac12 \nabla\cdot\vecc \ge 0 \quad\text{in }\Om\,,\\
\text{2) } &
\vecnu\cdot\vecc\ge 0 \quad\text{on } \wtGamma_{2,1}
:= \{ x \in \wtGamma_2 \,|\, \talpha(x) = \vecnu\cdot\vecc\}\,,
\hfill\rule{23ex}{0ex}\\
\text{3) } &
\talpha-\frac12\vecnu\cdot\vecc\ge 0
\quad\text{on } \wtGamma_2\setminus\wtGamma_{2,1}\,,
\quad
\vecnu\cdot\vecc\le 0\quad\text{on } \wtGamma_1\,.
\end{aligned}
\end{equation}
These assumptions together with \eqref{eq:Kspd} ensure that the bilinear form $a$ is
at least positively semidefinite on $V$, as the following identity shows:
\begin{align*}
a(v,v)&=(\vecK\nabla v,\nabla v) + (r+ \frac12 \nabla\cdot\vecc,v^2)\\
&-(\vecnu\cdot\vecc,v^2)_{\wtGamma_1} + \frac12(\vecnu\cdot\vecc,v^2)_{\wtGamma_{2,1}}
+(\talpha-\frac12\vecnu\cdot\vecc,v^2)_{\wtGamma_2\setminus\wtGamma_{2,1}}
\quad\text{for all } v \in V.
\end{align*}
\begin{remark}
The above choice of boundary conditions \eqref{eq2020:kap_32_334_ofdivergence}
together with the requirements \eqref{eq:coercivityconddviform},
especially the sign condition to $\vecnu\cdot\vecc$ on $\wtGamma_1$ in \eqref{eq:coercivityconddviform},3),
does not include the possibility to prescribe boundary data for
$(\vecK \nabla u - \vecc\,u ) \cdot \vecnu$ at an outflow boundary.

At first glance this seems physically questionable, but that there are arguments
underlying this fact in two extreme situations, namely the diffusion-dominated
and the convection-dominated regimes.
First we note that the boundary term on $\wtGamma_1$ in the above identity can be controlled
thanks to the boundedness of the trace operator $V\to L^2(\wtGamma_1)$
\cite[Thm.~1.6.6]{Brenner:08}.
That is, if $k_0$ in \eqref{eq:Kspd} is sufficiently large 
in comparison with the $L^\infty$-norm of $\vecnu\cdot\vecc$ on $\wtGamma_1$
(``diffusion-dominated regime near $\wtGamma_1$''),
the positive definiteness of the bilinear form $a$ on $V$ can be preserved
even if $\vecnu\cdot\vecc> 0$ on $\wtGamma_1$.
In the contrary case, if the $L^\infty$-norm of $\vecK$ on $\Om$ is
very small in comparison with the $L^\infty$-norm of $\vecnu\cdot\vecc$ on $\wtGamma_1$
(``convection-dominated regime near $\wtGamma_1$''),
the problem is almost elliptically degenerate, and in such a case it
is not appropriate to prescribe boundary data at an outflow (``noncharacteristic'') boundary.
In practice, a so-called do-nothing boundary condition, which is more or less artificial,
is often prescribed at an outflow boundary in order to avoid boundary layer effects.
Therefore, in the convection-dominated case, outflow boundary conditions
can and have to be be modeled via $\wtGamma_2$.
\end{remark}
In the next step, to formulate the needed regularity conditions to
the problem \eqref{eq:-1.33lin}--\eqref{eq2020:kap_32_334_ofdivergence},
and to describe the discretization, we introduce a family $(\Th)_h$ of consistent partitions
of the domain $\Om$.
Given an \emph{admissible partition} $\Tnoh:=\Th$ of $\Om$
(in the sense of \cite[(FEM 1)]{Ciarlet:02b}, where we omit the subscript $h$
for simplicity of notation),
it is called \emph{consistent}, if the following additional properties are met:
\begin{itemize}
\item
Every face $F$ of an element $K \in\Tnoh$ is either a subset of the boundary $\pOm$ of $\Om$
or identical to a face of another element $K'\in\Tnoh$.
\item
Each of the boundary subsets $\wtGamma_j$
is interrelated with a set of faces $\mathcal{F}_j$ in the following way:
\[
\text{cl}_\text{\,rel}\big(\wtGamma_j\big)
= \bigcup_{F \in \mathcal{F}_j} F, \quad j\in\{1,2,3\},
\]
where $\text{cl}_\text{\,rel}$ denotes the closure of a boundary subset
in the relative topology of $\pOm$.
\end{itemize}
To simplify the notation in the further analysis, we denote the set of all faces
of all elements of $\Tnoh$ by $\overline{\mathcal{F}}$,
the set of those faces that are lying on the boundary $\pOm$
by $\partial\mathcal{F}$, and the set of all faces of an element $K\in\Tnoh$ by $\mathcal{F}_K$.
Hence $\mathcal{F}:=\overline{\mathcal{F}}\setminus\partial\mathcal{F}$ is the set
of all interior faces.

Furthermore we introduce \emph{jumps} and \emph{averages} of piecewise defined functions
as follows.
Let $F\in\mathcal{F}_K\cap\mathcal{F}_{K'}\ne\emptyset$
be an interior face in the partition $\Tnoh$,
$K$, $K' \in\Tnoh$, $K\ne K'$.
By $\vecnu_K$ we denote the outer unit normal on the boundary $\partial K$
of an element $K\in\Tnoh$.
In case of scalar functions $v:\;\overline{\Om}\to\R$
such that $v|_{K} \in H^1(K)$, $v|_{K'} \in H^1(K')$,
we define
\begin{equation}\label{eq:scaljumpavg}
\begin{alignedat}{2}
\jump{v}
&:= \jump{v}_{F} &\,:=\,& v|_{K} \, \vecnu_K + v|_{K'} \, \vecnu_{K'}\,, \\
\avg{v}
&:= \avg{v}_{F} &\,:=\,& \frac12( v|_{K} + v|_{K'} )\,,
\end{alignedat}
\end{equation}
where here and in the subsequent relationships \eqref{eq:intjumpavg}--\eqref{eq:magicf}
the terms on right-hand sides are to be understood in the sense of traces on the face $\mathcal{F}$.
In case of vector fields $\vecp:\;\overline{\Om}\to\R^d$ with
$\vecp|_{K}\in H(\div;K)$ and $\vecp|_{K'}\in H(\div;K')$,
we set
\begin{equation}\label{eq:intjumpavg}
\begin{alignedat}{2}
\jump{\vecp} &:= \jump{\vecp}_{F}
&\,:=\,& \vecp|_{K} \cdot \vecnu_K + \vecp|_{K'} \cdot \vecnu_{K'}\,, \\
\avg{\vecp} &:= \avg{\vecp}_{F}
&\,:=\,& \frac12( \vecp|_{K} + \vecp|_{K'} ) \,.
\end{alignedat}
\end{equation}
If $F$ is a boundary face, \ie, $F\in\partial\mathcal{F}\cap\mathcal{F}_K$ for some $K \in\Tnoh$,
we define
\begin{equation}\label{eq:bdjumpavg}
\jump{v}_{F} := v|_{K} \, \vecnu_{K},
\quad
\avg{v}_{F} := v|_{K},
\quad
\jump{\vecp}_{F} := \vecp|_{K} \cdot \vecnu_K,
\quad
\avg{\vecp}_{F} := \vecp|_{K}.
\end{equation}
The definitions \eqref{eq:scaljumpavg}--\eqref{eq:bdjumpavg} are designed in such a way that
the averages retain the function type, while jumps of scalar functions are vector fields and vice versa.

A very useful relation between jumps and averages on interior faces is the so-called
\emph{magic formula}:
\begin{equation}\label{eq:magicf}
\avg{v}\jump{\vecp}+\jump{v}\avg{\vecp}
= v|_{K}\,\vecp|_{K} \cdot \vecnu_K + v|_{K'}\,\vecp|_{K'} \cdot \vecnu_{K'}.
\end{equation}
Finally, for $k \in \N$, $q\ge 1$, we define the \emph{broken Sobolev space} $W^{k,q}(\Tnoh)$
on a partition $\Tnoh$ of the domain $\Om$ by
\[
W^{k,q}(\Tnoh) := \{ v \in L^2(\Om) \;|\; v|_{K} \in W^{k,q}(K)
\quad\text{for all }K \in\Tnoh \},
\]
equipped with the norm
\[
\|v\|_{k,q,\Tnoh} := \Big( \sum_{K \in\Tnoh} \|v\|_{k,q,K}^q \Big)^{1/q}
\quad\text{for } q\in [1,\infty)
\]
or
\[
\|v\|_{k,\infty,\Tnoh} := \max_{K \in\Tnoh} \|v\|_{k,\infty,K}
\]
respectively.
As usual, by
\[
|v|_{k,q,\Tnoh} := \Big( \sum_{K \in\Tnoh} |v|_{k,q,K}^q \Big)^{1/q}
\quad\text{for } q\in [1,\infty)
\]
or
\[
|v|_{k,\infty,\Tnoh} := \max_{K \in\Tnoh} |v|_{k,\infty,K},
\]
resp., the corresponding seminorms are denoted.
In the case $q=2$ we use the standard notations $H^k(\Tnoh):=W^{k,2}(\Tnoh)$,
$\|v\|_{k,\Tnoh}:=\|v\|_{k,2,\Tnoh}$, $|v|_{k,\Tnoh}:=|v|_{k,2,\Tnoh}$.

\subsubsection*{Regularity assumptions}
Basically, we assume that the problem has a unique weak solution $u \in H^1(\Om)$.
This can be guaranteed by making additional assumptions to \eqref{eq:coercivityconddviform},
see, \eg, \cite[Thm.~3.16]{Angermann:21a}.
To analyze the consistency error, however, we need additional regularity assumptions
(which are also additional requirements to the data of
\eqref{eq:-1.33lin}--\eqref{eq2020:kap_32_334_ofdivergence}):
\begin{equation}\label{eq:7.17a}
\vecK\nabla u \in H(\div;\Om),\ cu\in H^1(\Om).
\end{equation}
Note that, as consequence of \eqref{eq:7.17a},
\begin{equation}\label{eq:7.17c}
\begin{aligned}
\jump{\vecK\nabla u - \vecc u}_F= 0
\quad\text{for all } F\in\mathcal{F},
\end{aligned}
\end{equation}
\ie, the normal components of the flux densities are continuous across the inner element boundaries.

\section{Example I: The Crouzeix-Raviart discretization}
\label{sec:NonconformingFE}
In this section we consider \emph{shape-regular} (\ie, \emph{regular}
in the sense of \cite[Sect.~3.1.1]{Ciarlet:02b}) families of consistent simplicial partitions of $\Om$.
To specify of the approximation spaces we introduce the space
\begin{equation}\label{eq:defCR1}
CR_1(\Om):= \{ v \in L^{1}(\Om) \;|\; v|_{K} \in \mathcal{P}_1(K)
\text{ and }
\int_F \jump{v} \, d\sigma = \bzero
\text{ for all } F \in \mathcal{F} \},
\end{equation}
where $\mathcal{P}_1(K)$ denotes the set of polynomials of degree one on $K$.
This space is is also known as the global \emph{Crouzeix-Raviart space} of degree one.

It should be noted that the an element $v \in V_h$ is bi-valued on $\mathcal{F}$ in general.
With that we define
\begin{equation}\label{eq:defV0h}
U_h := V_h := \Big\{ v\in CR_1(\Om) \;|\; \int_F v \, d\sigma = 0
\quad\text{for all } F \in \mathcal{F}_3 \Big\}.
\end{equation}
Now, to formulate a (nonconsistent) discretization of
\eqref{eq:-1.33lin}--\eqref{eq2020:kap_32_334_ofdivergence},
we introduce the forms
\begin{equation}\label{eq2020:kap81_823}
\begin{split}
a_h(w,v) &:= \sum_{K \in\Tnoh} (\vecK \nabla w - \vecc w, \nabla v)_K + (rw,v)
+ \sum_{F \in \mathcal{F} \cup \mathcal{F}_3} (\cupw{w},\jump{v})_F
+ (\talpha w, v)_{\wtGamma_2},\\
\ell_h(v) &:= \ell(v)
\qquad\text{for all }w,v \in U(h):=V(h):=H^1(\Tnoh),
\end{split}
\end{equation}
where the \emph{upwind evaluation} $\cupw{w}$ of the term $\vecc w$
at the interior faces $F\in\mathcal{F}_K\cap\mathcal{F}_{K'}\ne\emptyset$
is defined pointwise as
\begin{equation}\label{eq2020:kap81_824}
\begin{aligned}
\cupw{w} := \begin{cases}
\vecc w|_{K} & \text{for } \vecc \cdot \vecnu_{K} > 0, \\
\vecc w|_{K'} & \text{for } \vecc \cdot \vecnu_{K} \le 0.
\end{cases}
\end{aligned}
\end{equation}
At the boundary faces $F\in\partial\mathcal{F}$ of a simplex $K$ we set
\begin{equation}\label{eq2020:kap81_825}
\begin{aligned}
\cupw{w} := \begin{cases}
\vecc w|_{K} & \text{for } \vecc \cdot \vecnu_{K} > 0, \\
0 & \text{for } \vecc \cdot \vecnu_{K} \le 0.
\end{cases}
\end{aligned}
\end{equation}
Then the discrete method reads as follows:
\begin{equation}\label{eq2020:kap81_817}
\text{Find } u_h \in V_h \text{ such that } \quad
a_h (u_h,v_h) = \ell_h(v_h)
\quad\text{for all } v_h \in V_h \,.
\end{equation}
At first we study the coercivity of the bilinear form $a_h$.
So let $v\in H^1(\Tnoh)$. Starting from
\begin{align*}
a_h(v,v)
&= \sum_{K \in\Tnoh} (\vecK \nabla v, \nabla v)_K
+ (rv,v)
+ (\talpha v, v)_{\wtGamma_2}\\
&\qquad - \sum_{K \in\Tnoh} (\vecc v, \nabla v)_K
+ \sum_{F \in \mathcal{F} \cup \mathcal{F}_3} (\cupw{v},\jump{v})_F,
\end{align*}
we treat the fourth term as follows :
\begin{align*}
- \sum_{K \in\Tnoh} (\vecc v, \nabla v)_K
&= -\frac12 \sum_{K \in\Tnoh} (\vecc, \nabla v^2)_K\\
&= \frac12 \sum_{K \in\Tnoh} ((\nabla\cdot\vecc)v,v)_K
- \frac12 \sum_{F \in \partial\mathcal{F}} (\vecc v,\vecnu v)_F
- \frac12 \sum_{F \in \mathcal{F}}  (\vecc, \jump{v^2})_F.
\end{align*}
This gives
\begin{align*}
a_h(v,v)
&= \sum_{K \in\Tnoh} (\vecK \nabla v, \nabla v)_K
+ \Big(\Big(r+\frac12\nabla\cdot\vecc\Big) v,v\Big)
+ (\talpha v, v)_{\wtGamma_2}
- \frac12 \sum_{F \in \partial\mathcal{F}} (\vecc v,\vecnu v)_F\\
&\quad + \sum_{F \in \mathcal{F}} \Big[(\cupw{v},\jump{v})_F-\frac12 (\vecc, \jump{v^2})_F\Big]
+ \sum_{F \in \mathcal{F}_3} (\cupw{v},\jump{v})_F\\
&\ge \sum_{K \in\Tnoh} k_0\|\nabla v\|_K^2
- \frac12 \sum_{F \in \mathcal{F}_1} (\vecc v,\vecnu v)_F
+ \sum_{F \in \mathcal{F}_2} \Big(\Big(\talpha- \frac12 \vecnu\cdot\vecc\Big)  v, v\Big)_F\\
&\quad + \sum_{F \in \mathcal{F} \cup \mathcal{F}_3}
\Big[(\cupw{v},\jump{v})_F-\frac12 (\vecc, \jump{v^2})_F\Big]\\
&\ge \sum_{K \in\Tnoh} k_0\|\nabla v\|_K^2
+ \frac12 \sum_{F \in \mathcal{F} \cup \mathcal{F}_3} (|\vecc \cdot \vecnu |,\jump{v}^2)_F
\ge k_0|v|_{1,\Tnoh}^2,
\end{align*}
where we have used the properties \eqref{eq:coercivityconddviform},
\eqref{eq2020:kap81_824}, and \eqref{eq2020:kap81_825}.

If we now include additional conditions
(which are similar to the conditions mentioned at the beginning of the subsection about
the regularity assumptions, but a little more stringent),
we obtain the $(V+V_h)$-coercivity of $a_h$ uniform in $h$.
Namely, we assume that, in addition to the conditions \eqref{eq:coercivityconddviform},
one of the following conditions is satisfied:
\begin{align}\label{eq:addcoerccond1}
\text{a)}\quad & |\wtGamma_3|_{d-1} >0.\\
\text{b)}\quad & \text{There exists some constant } r_0>0 \text{ such that }
r + \frac12 \nabla\cdot\vecc \ge r_0 \text{ on } \Om.
\label{eq:addcoerccond2}
\end{align}
Indeed, the case b) immediately implies the estimate
$a_h(v,v) \ge k_0 |v|_{1,\Tnoh}^2 + r_0 \| v\|_{0,\Om}^2\ge \min\{k_0;r_0\}\|v\|_{1,\Tnoh}^2$
for all $v\in H^1(\Tnoh)$.

To verify the case a) for $v\in V$ we make use of
the Poincar\'{e}-Friedrichs inequality \cite[Exercise 5.x.13]{Brenner:08},
\ie, there exists a constant $C_\mathrm{PF} > 0$ such that
\[
\|v\|_0 \le C_\mathrm{PF} |v|_1 = C_\mathrm{PF} |v|_{1,\Tnoh}
\quad \text{for all } v \in V.
\]
If $v_h\in V_h$, we make use of a discrete version of this inequality
which can be proven analogously to the proof of \cite[Thm.~(10.6.12)]{Brenner:08}, \ie,
\[
\| v_h \|_0 \le \tilde{C}_\mathrm{PF} |v_h|_{1,\Tnoh}
\quad \text{for all } v_h \in V_h
\]
with some constant $\tilde{C}_\mathrm{PF} > 0$ independent of $h$.
Since $\|v+v_h\|_{V+V(h)}:=\inf\{\|v\|_{0,\Om}+\|v_h\|_{0,\Om}\}$ is a norm
of $v+v_h$ considered as an element of the subspace $V+V_h\subset L^2(\Om)$,
the combination of the above estimates shows that $|v|_{1,\Tnoh}$ itself is already
a norm on $V+V_h$.
Hence $a_h(v,v) \ge k_0 |v|_{1,\Tnoh}^2$ is the desired coercivity estimate,
which in turn implies that the inf-sup condition \eqref{EQ:NB2h} holds with $\alpha_h=k_0$.

Next we investigate the consistency error. For $u\in V\cap H^2(\Tnoh)$,
satisfying the regularity condition \eqref{eq:7.17a}
(this defines the regularity space $W$), and $v_h\in V_h$, we have:
\begin{equation}\label{eq:CRconserr}
\begin{aligned}
&a_h(u,v_h) - \ell_h(v_h)\\
&= \sum_{K \in\Tnoh} (\vecK \nabla u - \vecc u, \nabla v_h)_K
+ \sum_{F \in \mathcal{F} \cup \mathcal{F}_3} (\cupw{u},\jump{v_h})_F + (ru-f,v_h)\\
&\quad  - (\tg_1,v_h)_{\wtGamma_1}
+ (\talpha u - \tg_2, v_h)_{\wtGamma_2}.
\end{aligned}
\end{equation}
By integration by parts, the diffusion term can be rewritten as
\begin{align*}
\sum_{K \in\Tnoh} (\vecK\nabla u,\nabla v_h)_K
&= - \sum_{K \in\Tnoh} (\nabla\cdot(\vecK\nabla u),v_h)_K\\
&+ \sum_{F \in \mathcal{F}}
\big[(\avg{\vecK\nabla u},\jump{v_h})_F + (\jump{\vecK\nabla u},\avg{v_h})_F\big]
+\sum_{F \in \partial\mathcal{F}} (\vecnu\cdot\vecK\nabla u,v_h)_F,
\end{align*}
where we have used \eqref{eq:magicf}.
A rearrangement of the last three terms, taking into consideration
the definition \eqref{eq:bdjumpavg} of the boundary jumps and averages, yields
\begin{equation}\label{eq2020:kap81_821c}
\begin{aligned}
\sum_{K \in\Tnoh} (\vecK\nabla u,\nabla v_h)_K
&= - \sum_{K \in\Tnoh} (\nabla\cdot(\vecK\nabla u),v_h)_K\\
&\quad + \sum_{F \in \mathcal{F} \cup \mathcal{F}_3}(\avg{\vecK\nabla u},\jump{v_h})_F
+ \sum_{F \in \mathcal{F} \cup \mathcal{F}_1 \cup \mathcal{F}_2}(\jump{\vecK\nabla u},\avg{v_h})_F.
\end{aligned}
\end{equation}
The next terms to consider are the convection terms. Using integration by parts
in the terms over $K \in\Tnoh$ together with \eqref{eq:magicf}, we get
\begin{equation}\label{eq2020:kap81_828}
\begin{aligned}
&-\sum_{K \in\Tnoh} (\vecc u, \nabla v_h)_K
+ \sum_{F \in \mathcal{F} \cup \mathcal{F}_3} (\cupw{u},\jump{v_h})_F\\
&=(\nabla\cdot(\vecc u), v_h)
+ \sum_{F \in \mathcal{F} \cup \mathcal{F}_3} (\cupw{u},\jump{v_h})_F\\
&\quad -\sum_{F \in \mathcal{F}}
\big[(\avg{\vecc u},\jump{v_h})_F + (\jump{\vecc u},\avg{v_h})_F\big]
- \sum_{F \in \partial{F}} (\vecc u,\vecnu v_h)_F\\
&=(\nabla\cdot(\vecc u), v_h)
+ \sum_{F \in \mathcal{F} \cup \mathcal{F}_3} (\cupw{u}-\avg{\vecc u},\jump{v_h})_F\\
&\quad -\sum_{F \in \mathcal{F}} (\jump{\vecc u},\avg{v_h})_F
- \sum_{F \in \mathcal{F}_1 \cup \mathcal{F}_2} (\vecc u,\vecnu v_h)_F.
\end{aligned}
\end{equation}
Since $\jump{u}=\bzero$ on $\mathcal{F}$ thanks to $u \in H^1(\Om)$,
see \cite[Lemma 1.23]{DiPietro:12},
we have
\[
\sum_{F \in \mathcal{F} \cup \mathcal{F}_3} (\cupw{u}-\avg{\vecc u},\jump{v_h})_F
=\sum_{F \in \mathcal{F}_3} (\cupw{u}-\vecc u,\vecnu v_h)_F\,.
\]
The latter term vanishes because of $u=0$ on $\wtGamma_3$.
Inserting \eqref{eq2020:kap81_821c}, \eqref{eq2020:kap81_828} into \eqref{eq:CRconserr},
we arrive at
\begin{align*}
&a_h(u,v_h) - \ell_h(v_h)\\
&= - \sum_{K \in\Tnoh} (\nabla\cdot(\vecK\nabla u),v_h)_K
+ \sum_{F \in \mathcal{F}}(\jump{\vecK\nabla u - \vecc u},\avg{v_h})_F\\
&\quad
+ \sum_{F \in \mathcal{F} \cup \mathcal{F}_3}(\avg{\vecK\nabla u},\jump{v_h})_F
+ \sum_{F \in \mathcal{F}_1 \cup \mathcal{F}_2}(\jump{\vecK\nabla u},\avg{v_h})_F\\
&\quad
+ (\nabla\cdot(\vecc u), v_h)
- \sum_{F \in \mathcal{F}_1 \cup \mathcal{F}_2} (\vecc u,\vecnu v_h)_F\\
&\quad + (ru-f,v_h) - (\tg_1,v_h)_{\wtGamma_1}
+ (\talpha u - \tg_2, v_h)_{\wtGamma_2}\\
&=( -\nabla\cdot(\vecK\nabla u - \vecc u) + ru -f, v_h)
+ \sum_{F \in \mathcal{F} \cup \mathcal{F}_3}(\avg{\vecK\nabla u},\jump{v_h})_F\\
&\quad + \sum_{F \in \mathcal{F}_1 \cup \mathcal{F}_2} (\vecK\nabla u-\vecc u,\vecnu v_h)_F
- (\tg_1,v_h)_{\wtGamma_1}
+ (\talpha u - \tg_2, v_h)_{\wtGamma_2},
\end{align*}
where we have used \eqref{eq:7.17c}.
The first term and the sum of the last three terms vanish since the differential equation
\eqref{eq:-1.33lin} and the boundary conditions \eqref{eq2020:kap_32_334_ofdivergence}
on $\wtGamma_1 \cup\wtGamma_2$ are satisfied as equations in $L^2(\Om)$
and $L^2(\wtGamma_1 \cup\wtGamma_2)$, respectively.
Thus we get
\begin{align*}
a_h (u, v_h) - \ell_h(v_h)
&= \sum_{F \in \mathcal{F} \cup \mathcal{F}_3}(\avg{\vecK\nabla u},\jump{v_h})_F\\
&= \sum_{F \in \mathcal{F}}\Big(\avg{\vecK\nabla u}
-\frac12\big[\Pi_{K}(\vecK\nabla u)+\Pi_{K'}(\vecK\nabla u)\big],\jump{v_h}\Big)_F\\
&\quad + \sum_{F \in \mathcal{F}_3}\Big(\vecK\nabla u-\Pi_{K}(\vecK\nabla u),v_h\Big)_F,
\end{align*}
where
\[
\Pi_{K}(\vecK\nabla u) := \frac{1}{|K|}(\vecK\nabla u,1)_K
\quad\text{for all }
K\supset F.
\]
Thanks to the weak continuity condition in the definition \eqref{eq:defCR1} of $CR_1(\Om)$
and the weak homogeneous Dirichlet boundary condition (cf.~\eqref{eq:defV0h}),
the newly added terms do not change anything.
Therefore, by the Cauchy--Schwarz--Bunyakovsky inequality,
\begin{equation}\label{eq2020:kap81_822}
|a_h (u, v_h) - \ell_h(v_h)|
\le \Big( \sum_{F \in \mathcal{F} \cup \mathcal{F}_3}
h_{F} \big\| \vecK\nabla u - \Pi_{K}( \vecK\nabla u ) \big\|_{0,F}^2 \Big)^{1/2}
\Big( \sum_{F \in \mathcal{F} \cup \mathcal{F}_3}
h_{F}^{-1} \big\| \jump{v_h} \big\|_{0,F}^2 \Big)^{1/2}.
\end{equation}
The multiplicative trace inequality \cite[Lemma 2.19]{Dolejsi:15}
\[
\|v\|_{0, \partial K}^2 \le C \big[\|v\|_{0,K}|v|_{1,K} + h_K^{-1} \|v\|_{0,K}^2\big]
\quad\text{for all }v \in H^{1} (K),\ K \in\Tnoh,
\]
a standard error estimate for $L^2$-projections (see, \eg, \cite[Lemma 2.24]{Dolejsi:15}),
and the relationship
\begin{equation}\label{eq:constantsindependent-1}
C_{\mathcal{F}}^{-1} h_K \le h_{F} \le C_{\mathcal{F}} h_K
\quad\text{for all } F \in\mathcal{F}_K,\ K \in\Tnoh
\end{equation}
with a constant $C_{\mathcal{F}} > 0$ independent of $h_K$
allow to obtain the upper bound
\[
C h |\vecK\nabla u|_{1,\Tnoh}
\]
for the first factor in \eqref{eq2020:kap81_822}.

The second factor in \eqref{eq2020:kap81_822} can be treated as follows.
Denoting by $a_{S,F}$ the barycentre of the face $F$, we observe that
\[
\big\| \jump{v_h} \big\|_{0,F}
= \big\| \jump{v_h} - v_h(a_{S,F}) \big\|_{0,F}.
\]
Since both $\big(v_h - v_h(a_{S,F})\big)\big|_K$ and $\big(v_h - v_h(a_{S,F})\big)\big|_{K'}$
vanish at the same point in $F\in\mathcal{F}_K\cap\mathcal{F}_{K'}\ne\emptyset$,
the scaled trace inequality (see, \eg, \cite[Lemma 7.5]{Angermann:21a})
\[
\|v\|_{0,F} \le C h_K^{1/2} |v|_{1,K}
\quad\text{for all } v \in H^1(K)
\]
is applicable, leading together with \eqref{eq:constantsindependent-1}
to the following upper bound (up to a multiplicative constant) of the second factor:
\[
\Big( \sum_{F=K\cap K'\in\mathcal{F}}
\big[ h_K^{-1} h_K |v_h|_{1,K}^2 + h_{K'}^{-1} h_{K'} |v_h|_{1,K}^2\big]
+ \sum_{F=K\cap\wtGamma_3\in\mathcal{F}_3} h_K^{-1} h_K |v_h|_{1,K}^2 \Big)^{1/2}
\le C |v_h|_{1,\Tnoh}.
\]
Putting the obtained estimates together, we arrive at the following estimate
of the consistency error:
\begin{equation}\label{eq:CRconsistency}
|a_h (u, v_h) - \ell_h(v_h)|
\le C h |\vecK\nabla u|_{1,\Tnoh} |v_h|_{1,\Tnoh}.
\end{equation}
In order to be able to apply Theorem~\ref{th:Strang2},
the approximation order of $V_h$ remains to be determined.
Since the space $\widetilde{V}_h$ of conforming $\mathcal{P}_1$-elements is a subspace of $V_h$,
it follows, for a sufficiently smooth weak solution $u\in W\ (\subset H^2(\Tnoh))$
of \eqref{eq:-1.33lin}--\eqref{eq2020:kap_32_334_ofdivergence} that
\[
\inf_{v_h \in V_h} \| u - v_h \|_{1, \Tnoh}
\le \inf_{v_h \in \widetilde{V}_h} \| u - v_h \|_{1, \Tnoh}
\le Ch |u|_{2,\Tnoh}\,.
\]
In summary, we have proved the following result.
\begin{theorem}\label{th2020:kap81_shaperegular}
Let the family of triangulations be shape-regular,
the weak solution $u \in V\cap H^2(\Tnoh)$ be such that \eqref{eq:7.17a} is satisfied,
and the diffusion coefficient $\vecK$ be so smooth that $\vecK\nabla u \in H^1(\Tnoh)^d$.
Then, under the conditions \eqref{eq:coercivityconddviform}, \eqref{eq:addcoerccond1}
or \eqref{eq:coercivityconddviform}, \eqref{eq:addcoerccond2},
the following error estimate holds for the first-order Crouzeix-Raviart solution
$u_h\in V_h\subset CR_1(\Om)$ of the discrete problem \eqref{eq:defCR1}--\eqref{eq2020:kap81_817}:
\[
\|u - u_h\|_{1,\Tnoh} \le C h \big[|u|_{2,\Tnoh}+ |\vecK\nabla u|_{1,\Tnoh}\big],
\]
where the constant $C>0$ does not depend on $h$.
\end{theorem}
The error bound can be simplified if additional smoothness of the diffusion coefficient $\vecK$ is assumed.
\begin{corollary}\label{cor:CRconv}
In addition to the assumptions of Theorem~\ref{th2020:kap81_shaperegular},
let $\vecK\in W^{1,\infty}(\Tnoh)$.
Then, for the solution $u_h\in V_h\subset CR_1(\Om)$
of the discrete problem \eqref{eq:defCR1}--\eqref{eq2020:kap81_817}, the error estimate
\[
\|u - u_h\|_{1,\Tnoh} \le Ch |u|_{2,\Tnoh}
\]
with a constant $C>0$ independent of $h$ holds.
\end{corollary}
So we have seen that the effect of including inter-element boundary terms in
the discrete formulation is to guarantee the coercivity.
They have no influence on the consistency error.

\subsubsection*{Convergence order in a weaker norm}

In order to be able to apply Theorem~\ref{l:weakererror}, the adjoint problem \eqref{eq:specialadjoint}
and its discretization \eqref{eq:discrspecialadjoint} have to be investigated.
The adjoint problem \eqref{eq:specialadjoint} with $g\in Z:=L^2(\Om)$ is given by the forms
\begin{equation}\label{eq:adjoint1}
\begin{aligned}
a^\prime(v,w) &:= a(w,v)
= (\vecK \nabla w - \vecc w, \nabla v) + (rw,v)
+ (\talpha w, v)_{\wtGamma_2}\\
&= (\vecK\nabla v,\nabla w) - (\vecc\cdot\nabla v,w) + (rv,w)
+ (\talpha v,w)_{\wtGamma_2},\\
\tilde\ell(w)&:=(w,g)
\quad\text{for all } v,w \in V,
\end{aligned}
\end{equation}
hence the adjoint problem corresponds to the following formal boundary value problem
in nondivergence form:
\begin{align}
- \nabla \cdot \left( \vecK\nabla v \right) - \vecc \cdot \nabla v + rv &= g \quad\text{in } \Om,
\label{eq:strongadjointde}\\
\vecK\nabla v \cdot \vecnu &= 0
\quad\text{on } \wtGamma_1\,,
\nonumber\\
\vecK\nabla v \cdot \vecnu + \talpha v &= 0
\quad\text{on } \wtGamma_2\,,
\nonumber\\
v &= 0
\quad\text{on } \wtGamma_3\,.
\nonumber
\end{align}
Analogous to the continuous case, the discrete adjoint problem is defined
as the adjoint of the discrete problem \eqref{eq2020:kap81_823}:
\[
a_h^\prime(v,w) := a_h(w,v) := \sum_{K \in\Tnoh} (\vecK \nabla w - \vecc w, \nabla v)_K + (rw,v)
+ \sum_{F \in \mathcal{F} \cup \mathcal{F}_3} (\cupw{w},\jump{v})_F
+ (\talpha w, v)_{\wtGamma_2}.
\]
Obviously, the $V_h$-coercivity constant of $a_h^\prime$ is the same
as that of $a_h$.

The consistency error
\begin{align*}
&\quad a_h^\prime(v,w_h) - \tilde\ell(w_h)\\
&= \sum_{K \in\Tnoh} (\vecK \nabla w_h - \vecc w_h, \nabla v)_K + (rw_h,v)
+ \sum_{F \in \mathcal{F} \cup \mathcal{F}_3} (\cupw{w_h},\jump{v})_F
+ (\talpha w_h, v)_{\wtGamma_2} - (g,w_h)
\end{align*}
can be split into the consistency error of the symmetric part
(cf.\ \eqref{eq2020:kap81_821c} for the relevant diffusion term)
and the nonsymmetric part
\[
-\sum_{K \in\Tnoh} (\vecc w_h, \nabla v)_K
+ \sum_{F \in \mathcal{F} \cup \mathcal{F}_3} (\cupw{w_h},\jump{v})_F
=-(\vecc w_h, \nabla v)
+ \sum_{F \in \mathcal{F} \cup \mathcal{F}_3} (\cupw{w_h},\jump{v})_F.
\]
The first term is part of the strong form (in the sense of $L^2(\Om)$) of the differential equation,
whereas the second term vanishes. Therefore the consistency error estimation
is reduced to the estimation of the consistency error
of the symmetric part.
Thanks to symmetry, the estimate can be taken directly from \eqref{eq:CRconsistency}
in the proof of Theorem~\ref{th2020:kap81_shaperegular}.
Consequently, the error estimate of Theorem~\ref{th2020:kap81_shaperegular}
also applies to the adjoint problem.

In order to apply Theorem~\ref{l:weakererror} we have to assume that the adjoint problem
is regular in the sense that for any right-hand side $g \in L^2(\Om)$
a unique solution $v_g \in V \cap H^2(\Tnoh)$ of the adjoint boundary value problem
\eqref{eq:specialadjoint} exists and
a stability estimate of the form
\begin{equation}\label{eq:adjointregularity1}
| v_g |_{2,\Tnoh} \le \tilde{C} \| g\|_0 \quad \text{for given } g \in L^2(\Om),
\end{equation}
is satisfied with some constant $\tilde{C} > 0$ (\ie, \eqref{eq2020:kap61_6176} holds with $Y:=V \cap H^2(\Tnoh)$).
Then the first term in the estimate of Theorem~\ref{l:weakererror}
can be bounded from above by $Ch^2 |u|_{2,\Tnoh}$, and the fourth term vanishes.
The third term, a consistency error for the original problem,
can be estimated analogously to \eqref{eq2020:kap81_822}
and the subsequent considerations, but with $v_h$ substituted by $v_{g} - v_{gh}$,
\ie, the approximation error of the
adjoint problems \eqref{eq:specialadjoint}, \eqref{eq:discrspecialadjoint}.

Thus the final estimate reads (compare \eqref{eq:CRconsistency} in the proof
of Theorem~\ref{th2020:kap81_shaperegular}):
\[
Ch |\vecK\nabla u|_{1,\Tnoh} \, |v_{g} - v_{gh}|_{1,\Tnoh}
\le Ch^2 |\vecK\nabla u|_{1,\Tnoh} \, |v_{g}|_{2,\Tnoh}
\le Ch^2 |\vecK\nabla u|_{1,\Tnoh} \|g\|_{0,\Om},
\]
which is the required relationship.

It remains to discuss the second term, a consistency error of the adjoint problem.
It can be treated similarly to the third term (with interchanged roles $u \leftrightarrow v_{g}$
and $v_{g} - v_{gh} \leftrightarrow u - u_h$), resulting in the bound
\[
Ch |\vecK\nabla v_g|_{1,\Tnoh} \, |u - u_h|_{1,\Tnoh}.
\]
But since we only have the estimate \eqref{eq:adjointregularity1}
of the $|\cdot|_{2,\Tnoh}$-seminorm of $v_g$, we need an additional assumption,
for instance a regularity requirement to $\vecK$ as in Corollary~\ref{cor:CRconv}.
Under this assumption we can apply Corollary~\ref{cor:CRconv} to estimate
both consistency errors.
In summary, we can formulate the following result.
\begin{theorem}
Let the family of triangulations be shape-regular,
$\vecK\in W^{1,\infty}(\Tnoh)$,
the weak solution $u \in V\cap H^2(\Tnoh)$ be such that \eqref{eq:7.17a} is satisfied,
and the solution of the adjoint problem \eqref{eq:specialadjoint} be regular such that
\eqref{eq:adjointregularity1} is satisfied.
Then, under the conditions \eqref{eq:coercivityconddviform}, \eqref{eq:addcoerccond1}
or \eqref{eq:coercivityconddviform}, \eqref{eq:addcoerccond2},
the first-order Crouzeix-Raviart solution $u_h\in V_h\subset CR_1(\Om)$
of the discrete problem \eqref{eq:defCR1}--\eqref{eq2020:kap81_817}
satisfies the error estimate
\[
\|u - u_h\|_{0,\Om} \le Ch^2 |u|_{2,\Tnoh},
\]
where $C>0$ is a constant independent of $h$.
\end{theorem}

\section{Example II: Discontinuous Galerkin methods}
\label{sec:DGM}
Based on the setting of Section \ref{sec:modelproblem},
here we consider more general consistent partitions of $\Om$
(not necessarily consisting of $d$-simplices only), namely
\emph{shape- and contact-regular} families of partitions, see \cite[Def.~1.38]{DiPietro:12}.
We also assume that all partitions are \emph{compatible} with
the structure of the boundary piece $\wtGamma_2$ in the sense that
there are subsets $\mathcal{F}_{2,1},\mathcal{F}_{2,2}\subset\mathcal{F}_{2}$
such that the following representation holds:
\begin{equation*}
\wtGamma_{2,1} = \bigcup_{F \in \mathcal{F}_{2,1}} F,
\quad
\wtGamma_2\setminus\wtGamma_{2,1} = \bigcup_{F \in \mathcal{F}_{2,2}} F\setminus\wtGamma_{2,1}.
\end{equation*}
This requirement is for clarity of presentation only. In principle, it can be omitted
if the correct integration regions, which then do not have to be complete faces,
are specified for the corresponding integrations.

We use the finite element spaces
\[
U_h:=V_h:= \mathbb{P}_k(\Tnoh) := \{ w_h \in L^2(\Om) \;|\; w_h|_K \in
\mathbb{P}_k(K) \quad\text{for all }K \in\Tnoh\},
\]
where $\mathbb{P}_k(K):=\mathcal P_k(K)$ or $\mathbb{P}_k(K):=\mathcal{Q}_k(K)$.
Here $\mathcal{Q}_k(K)$ denotes the set of tensor-product polynomials on $K$, which
is composed of $d$-variate polynomials of maximum partial degree $k$ with respect to each variable.
We also allow inhomogeneous Dirichlet boundary conditions
on $\wtGamma_3$, \ie, $\tg_3$ may be a nontrivial function.
For a fixed \emph{symmetrization parameter} $\theta \in \{0,\pm 1\}$
and a \emph{penalty parameter} $\mu > 0$,
the interior penalty discontinuous Galerkin method, in short \emph{IPG} method, reads as follows:

\medskip\quad
Find $u_h \in V_h$ such that
\begin{equation}\label{EQ:IPG_global}
a_h(u_h,v_h) = \ell_h(v_h) \quad\text{for all } v_h \in V_h,
\end{equation}
where
\begin{equation}\label{eq2020:kap_85_8631}
\begin{aligned}
a_h(u_h,v_h) &:= \sum_{K \in\Tnoh} \big[(\vecK \nabla u_h - \vecc u_h, \nabla v_h)_K + (ru_h,v_h)_K\big]\\
& + \sum_{F \in \mathcal{F} \cup \mathcal{F}_3}
\Big[\theta(\avg{\vecK \nabla v_h},\jump{u_h})_F
- \Big(\avg{\vecK \nabla u_h}
- \frac{\mu}{h_F}\jump{u_h},\jump{v_h}\Big)_F\Big]\\
& + \sum_{F \in \mathcal{F} \cup \mathcal{F}_{2,1}} (\cupw{u_h},\jump{v_h})_F
+ \sum_{F \in \mathcal{F}_{2,2}} (\talpha u_h, v_h)_F,\\
\ell_h(v_h) &:= (f,v_h)
+ \sum_{F \in \mathcal{F}_1} (\tg_1,v_h)_F
+ \sum_{F \in \mathcal{F}_2} (\tg_2,v_h)_F\\
& + \sum_{F \in \mathcal{F}_3} \Big[\frac{\mu}{h_F} (\tg_3, v_h)_F
+ (\theta \vecK \nabla v_h - \vecc v_h,\vecnu \tg_3)_F\Big].
\end{aligned}
\end{equation}
The parameter value $\theta = 0$ gives the \emph{incomplete IPG} (\emph{IIPG}) method, while
the choice $\theta = -1$ results in the \emph{symmetric IPG} (\emph{SIPG}).
The value $\theta = 1$ yields the \emph{nonsymmetric IPG} (\emph{NIPG}) method.
The artificial symmetrization term has no influence on the consistency properties
of the method (cf.\ the subsequent Lemma~\ref{LEM:stability_IPDG}),
nor does it generate an additional numerical flux on the interior element faces.
Based on an idea by Nitsche, the Dirichlet boundary conditions are weakly imposed.

Next we will show that the IPG methods can be characterized as consistent but nonconforming methods
(cf.\ Section~\ref{sec:GM_VarEq}).
Analogously to Section~\ref{sec:NonconformingFE}, we assume a sufficient regularity of the solution
of the continuous problem as in \eqref{eq:7.17a}.

\begin{lemma}
\label{LEM:stability_IPDG}
If $u\in V\cap H^2(\Tnoh)$ is the weak solution of the problem \eqref{eq:-1.33lin}--\eqref{eq2020:kap_32_334_ofdivergence}
satisfying the regularity condition \eqref{eq:7.17a}
(this defines the regularity space $W$), then, for $k\in\N$,
the above IPG methods are consistent, \ie,
\[
a_h(u,v_h) = \ell_h(v_h) \quad\text{for all } v_h \in V_h.
\]
\end{lemma}
\begin{proof}
Using the property $\jump{u}_F=\bzero$ on $F\in\mathcal{F}$,
it is not difficult to see that it holds, for an arbitrary test function $v_h \in V_h$:
\begin{align*}
&a_h(u,v_h) - \ell_h(v_h)\\
= & \sum_{K \in\Tnoh} \big[(\vecK \nabla u - \vecc u, \nabla v_h)_K + (ru,v_h)_K\big]
- \sum_{F \in \mathcal{F} \cup \mathcal{F}_3}(\avg{\vecK \nabla u},\jump{v_h})_F\\
& + \sum_{F \in \mathcal{F}_3}
\Big[\theta(\vecK \nabla v_h,\vecnu u)_F
+ \Big(\frac{\mu}{h_F}u,v_h\Big)_F\Big]\\
& + \sum_{F \in \mathcal{F} \cup \mathcal{F}_{2,1}} (\vecc u,\jump{v_h})_F
+ \sum_{F \in \mathcal{F}_{2,2}} (\talpha u, v_h)_F\\
& -(f,v_h)
- \sum_{F \in \mathcal{F}_1} (\tg_1,v_h)_F
- \sum_{F \in \mathcal{F}_2} (\tg_2,v_h)_F\\
& - \sum_{F \in \mathcal{F}_3} \Big[\frac{\mu}{h_F} (\tg_3, v_h)_F
- (\theta \vecK \nabla v_h - \vecc v_h,\vecnu \tg_3)_F\Big].
\end{align*}
Furthermore, since $u = \tg_3$ on $F \in \mathcal{F}_3$, we get
\begin{align*}
&a_h(u,v_h) - \ell_h(v_h)\\
= & \sum_{K \in\Tnoh} \big[(\vecK \nabla u - \vecc u, \nabla v_h)_K + (ru,v_h)_K\big]
- \sum_{F \in \mathcal{F} \cup \mathcal{F}_3}(\avg{\vecK \nabla u},\jump{v_h})_F\\
& + \sum_{F \in \mathcal{F} \cup \mathcal{F}_{2,1}} (\vecc u,\jump{v_h})_F
+ \sum_{F \in \mathcal{F}_{2,2}} (\talpha u, v_h)_F\\
& -(f,v_h)
- \sum_{F \in \mathcal{F}_1} (\tg_1,v_h)_F
- \sum_{F \in \mathcal{F}_2} (\tg_2,v_h)_F
+ \sum_{F \in \mathcal{F}_3} (\vecc v_h,\vecnu \tg_3)_F.
\end{align*}
The elementwise integration by parts (cf.~\eqref{eq2020:kap81_821c}, \eqref{eq2020:kap81_828})
yields, using the boundary conditions \eqref{eq2020:kap_32_334_ofdivergence}:
\begin{align*}
a_h(u,v_h)& - \ell_h(v_h)
= (- \nabla \cdot \left( \vecK \nabla u - \vecc u \right) + ru - f,v_h)\\
+ & (\left(\vecK\nabla u-\vecc u\right)\cdot\vecnu-\tg_1,v_h)_{\wtGamma_1}
+ (\left(\vecK\nabla u-\vecc u\right)\cdot\vecnu+\talpha u-\tg_2,v_h)_{\wtGamma_2}.
\end{align*}
The right-hand side vanishes since the differential equation \eqref{eq:-1.33lin}
is satisfied in the sense of $L^2(\Om)$,
and the boundary conditions \eqref{eq2020:kap_32_334_ofdivergence}
in the sense of the corresponding trace spaces.
\end{proof}
\begin{remark}[Interrelation with conventional FEM and Crouzeix-Raviart elements]
\label{REM:relations_DG}
\rule{0ex}{1ex}

\begin{enumerate}[1)]
\item
The use of the IPG bilinear and linear forms \eqref{eq2020:kap_85_8631}
together with conventional (continuous) finite element spaces $V_h \cap C^0(\overline{\Om})$
gives a conventional FEM with a different treatment of boundary conditions, since
all jumps on interfaces vanish due to the continuity of the ansatz and test functions.
\item
Compared to Section~\ref{sec:NonconformingFE}, the term
\[
\sum_{F\in\mathcal{F}\cup\mathcal{F}_3} (\vecK\nabla u,\jump{v_h})_F
\]
disappears from the consistency error representation
since the expression
\[
-\sum_{F\in\mathcal{F}\cup\mathcal{F}_3} (\avg{\vecK\nabla u_h},\jump{v_h})_F
\]
has been included in the bilinear form.
\end{enumerate}
\end{remark}

\subsubsection*{Stability of IPG methods}

In the next step we will demonatrate that a suitable norm $\|\cdot\|_{V_h}$ of energy type can be found
with respect to which the bilinear form $a_h$ is uniformly bounded and coercive.
Then the problem \eqref{EQ:IPG_global} can be solved uniquely in a stable manner.
A natural starting point is the NIPG method (\ie, $\theta = 1$),
since it contains comparatively few summands, which can be recasted
in such a way that finally the desired coercivity results.
For $v\in H^1(\Tnoh)$, we have the identity
\begin{equation}\label{eq:dGnorm}
\begin{aligned}
a_h^\text{NIPG}(v,v)
= & \sum_{K \in\Tnoh} \| \vecK^{1/2} \nabla v\|^2_{0,K}
+ \sum_{F \in \mathcal{F} \cup \mathcal{F}_3} \frac{\mu}{h_F}\|\jump{v}\|^2_{0,F}\\
& + \frac12 (2r + \nabla\cdot\vecc,v^2)
+ \frac12 \sum_{F \in \overline{\mathcal{F}} \setminus \mathcal{F}_{2,2}}
(| \vecc \cdot \vecnu |, \jump{v}^2)_F \\
& + \frac12\sum_{F \in \mathcal{F}_{2,2}}(2\talpha - \vecnu \cdot \vecc, v^2)_F
=: \|v\|^2_{V_h}
\end{aligned}
\end{equation}
(note that $\mu > 0$). Indeed, we can write:
\begin{align*}
& -\sum_{K \in\Tnoh} (v,\vecc\cdot\nabla v)_K
+ \sum_{F \in \mathcal{F} \cup \mathcal{F}_{2,1}} (\cupw{v},\jump{v})_F\\
= & -\frac12 \sum_{K \in\Tnoh} (\vecc,\nabla v^2)_K
+ \sum_{F \in \mathcal{F} \cup \mathcal{F}_{2,1}}
(\cupw{v},\jump{v})_F\\
= & \frac12 \sum_{K \in\Tnoh} (\nabla\cdot\vecc, v^2)_K
- \frac12 \sum_{F \in \mathcal{F}_1 \cup \mathcal{F}_{2,2} \cup \mathcal{F}_3}
(\vecc \cdot \vecnu, v^2)_F\\
& + \sum_{F \in \mathcal{F} \cup \mathcal{F}_{2,1}}
\Big[ (\cupw{v},\jump{v})_F - \frac12 (\vecc,\jump{v^2})_F\Big]\\
= & \frac12 \sum_{K \in\Tnoh} (\nabla\cdot\vecc, v^2)_K
+ \frac12 \sum_{F \in \overline{\mathcal{F}} \setminus \mathcal{F}_{2,2}}
(| \vecc \cdot \vecnu |, \jump{v}^2)_F
- \frac12 \sum_{F \in \mathcal{F}_{2,2}} (\vecc \cdot \vecnu, v^2)_F.
\end{align*}
The last equality is obtained as in the treatment of \eqref{eq2020:kap81_828},
using the sign conditions \eqref{eq:coercivityconddviform},
and the additional condition
\begin{equation}\label{eq:coercivityconddviform2}
\vecc \cdot \vecnu \le 0
\quad\text{on }\wtGamma_3.
\end{equation}

\begin{lemma}\label{lem2020:kap_85_816'}
Assume that, in addition to the conditions \eqref{eq:coercivityconddviform}
and \eqref{eq:coercivityconddviform2},
one of the conditions \eqref{eq:addcoerccond1} or \eqref{eq:addcoerccond2}
is satisfied.
Then, if $h>0$ is sufficiently small, \eqref{eq:dGnorm} defines a norm on $V+V_h$,
and there exists a constant $C > 0$ independent of $h$ such that
\[
\|v_h\|_{1, \Tnoh} \le C \|v_h\|_{V_h}
\quad \text{for all } v_h \in V+V_h.
\]
\end{lemma}
\begin{proof}
The nonnegativity of all terms in \eqref{eq:dGnorm} immediately yields the estimate
\[
\|v\|_{V_h}^2 \ge \sum_{K \in\Tnoh} \| \vecK^{1/2} \nabla v \|^2_{0,K}
\ge k_0 | v |^2_{1, \Tnoh},
\]
which shows that $\| \cdot \|_{V_h}$ is a seminorm on $H^1(\Tnoh)$.
Morever, the condition $\|v\|_{V_h} = 0$ implies that the element $v$ is piecewise constant.
The additional conditions together with the structure \eqref{eq:dGnorm} of $\|\cdot\|_{V_h}$
lead to $v=0$.

Indeed, as in Section~\ref{sec:NonconformingFE},
the case b) yields the estimate
$\|v\|_{V_h}^2 \ge \min\{k_0;r_0\}\|v\|_{1,\Tnoh}^2$ for all $v \in H^1(\Tnoh)$.
To prove a) we combine, as in Section~\ref{sec:NonconformingFE},
the Poincar\'{e}-Friedrichs inequality \cite[Exercise 5.x.13]{Brenner:08}
on $V$ with a discrete inequality on $V_h$.
Namely, as a consequence of \cite[Thm.~5.3]{DiPietro:12},
there is a constant $C_\mathrm{PI} > 0$ independent of $h > 0$
such that following refined Poincar\'{e} inequality holds:
\[
\| v_h \|_0 \le C_\mathrm{PI} \| v_h \|_{V_h}
\quad \text{for all } v_h \in V_h.
\]
\end{proof}
The following lemma summarizes the hitherto obtained properties of the bilinear and linear forms.
\begin{lemma}\label{LEM:DG_stability}
Assume that $\vecK$ is piecewise continuous, \ie, continuous on each $K\in\Tnoh$.
The bilinear form $a_h$ and the linear form $\ell_h$ are bounded on $V_h$ with
respect to $\| \cdot \|_{V_h}$, not necessarily uniform with respect to $h$.
If the condition
\begin{equation}\label{EQ:stab_condition2}
4\mu > (1 - \theta)|\mathcal{F}_K| C_\mathrm{tr}^2 \|\vecK\|_\infty,
\quad\text{where}\quad
|\mathcal{F}_K|:=\begin{cases}
d+1, & \mathbb{P}_k(K)=\mathcal P_k(K),\\
2d, &\mathbb{P}_k(K)=\mathcal{Q}_k(K)
\end{cases}
\end{equation}
is satisfied, the bilinear form $a_h$ is uniformly coercive with respect to $\| \cdot \|_{V_h}$
with the parameter $\alpha = \alpha_h$ independent of $h$ but depending on $\mu$.

The NIPG method is uniformly coercive with $\alpha = 1$ if $\mu > 0$.
\end{lemma}
\begin{proof}
The last statement follows immediately from the representation
\[
a_h(v_h,v_h) = \|v_h\|^2_{V_h}
- (1 - \theta) \sum_{F \in \mathcal{F} \cup \mathcal{F}_3}
(\avg{\vecK \nabla v_h},\jump{v_h})_F
\quad\text{for all }
v_h \in \mathbb{P}_k(\Tnoh).
\]
For $\theta\ne 1$, we proceed as follows.
The properties of $\vecK$, a discrete trace inequality \cite[Lemma~1.46]{DiPietro:12}
(here we need the shape- and contact-regularity of the family of partitions)
and Young's inequality with $\eps>0$
allow the estimation
\begin{align*}
|(\avg{\vecK \nabla v_h},\jump{v_h})_F|
&\le \| \avg{\vecK \nabla v_h} \|_{0,F} \|\jump{v_h} \|_{0,F} \\
&\le \|\vecK\|_\infty^{1/2} C_\mathrm{tr} h_F^{-1/2} \| \vecK^{1/2} \nabla v_h \|_{0,\Delta(F)}
\|\jump{v_h} \|_{0,F} \\
&\le \eps \| \vecK^{1/2} \nabla v_h \|^2_{0,\Delta(F)}
+ \frac{ C^2_\mathrm{tr} \|\vecK\|_\infty}{4\eps h_F} \| \jump{v_h}\|^2_{0,F},
\end{align*}
where $\Delta(F)$ denotes the union of the elements $K$ with face $F$.
Since every element $K$ has $|\mathcal{F}_K|$ faces, the first term
occurs at most $|\mathcal{F}_K|$ times after the summation.
So if $\eps$ is chosen such that $\eps (1 - \theta)|\mathcal{F}_K| < 1$,
the first two terms in \eqref{eq:dGnorm} absorb the respective terms in the above bound,
where condition \eqref{EQ:stab_condition2} is applied to the second term.

The boundedness (not necessarily uniform in $h$) is obvious since all bilinear and linear forms
on finite-dimensional spaces are bounded.
\end{proof}
Thus, the Lax-Milgram lemma ensures the existence of a unique solution to \eqref{EQ:IPG_global}.
\begin{remark}
\begin{enumerate}[1)]
\item
Lemma \ref{LEM:DG_stability} remains valid for certain nonsimplicial
and nonconsistent partitions provided that $|\mathcal{F}_K|$ is replaced
by the maximum number of faces of an element.
\item
Be means of more sophisticated techniques it is possible to show
that the NIPG method is also stable for $\mu = 0$. This procedure,
known as the \emph{OBB method}, goes back to Oden, Babu\v{s}ka, and Baumann \cite{Oden:98}.
\end{enumerate}
\end{remark}

In Lemma~\ref{LEM:DG_stability} it was already mentioned that the boundedness constants
may be $h$-dependent. This problem can be circumvented by finding a framework that allows
the application of Remark~\ref{folg:DG_analysis_conv},~2).
That is we try to specify a suitable normed space $(V(h), \|\cdot\|_{V(h)})$
in which $a_h$ is bounded.

\subsubsection*{Convergence analysis for the complete problem}

Assume that the weak solution $u$ of the problem \eqref{eq:-1.33lin}--\eqref{eq2020:kap_32_334_ofdivergence}
belongs to $V\cap H^2(\Tnoh)$.
In order to fulfill the assumptions of Remark~\ref{folg:DG_analysis_conv},~2),
we first construct a space $V(h) \supset\mathbb{P}_k(\Tnoh) + \spann(u)$
such that $\|\cdot\|_{V_h}$ is a norm on $V(h)$.
A suitable choice clearly is
\[
V(h) = H^2(\Tnoh) := \{ v \in L^2(\Om) \;|\; v \in H^2(K)
\quad\text{for all } K \in\Tnoh \}.
\]
However, on $H^2(\Tnoh)$ we cannot apply a discrete trace inequality
to control $\avg{\vecK \nabla u_h}_F$ on the faces $F \in \mathcal{F} \cup \mathcal{F}_3$
(cf.\ the proof of Lemma \ref{LEM:DG_stability}).

Therefore the norm $\| \cdot \|_{V_h}$ should be extended in such a way
that this term can also be controlled while retaining the uniform boundedness
of the (extended) bilinear form $a_h$ on $V(h) \times V_h$.
A possible choice motivated by this is
\begin{equation}\label{eq:defDGnorm}
\| w_h \|^2_{V(h)} := \| w_h \|^2_{V_h} + \sum_{F \in \mathcal{F} \cup
\mathcal{F}_3} \frac{h_F}{\mu} \| \avg{\vecK^{1/2} \nabla w_h} \|^2_{0,F}
+ \sum_{F \in \overline{\mathcal F}} \frac{h_F}{\mu} \| \vecc w_h \|^2_{0,F}
\quad\text{for all } w_h \in V(h),
\end{equation}
where the last term is to be understood in such a way that both one-sided traces
are taken into account (\ie, it is evaluated twice).
\begin{lemma}\label{LEM:continuity_in_orthogonal_spaces}
Let the assumptions of Lemma~\ref{lem2020:kap_85_816'} be satisfied.
Then there exists a constant $\widetilde{M}_h>0$ bounded in $h$ such that
\[
a_h(w_h,v_h) \le \widetilde{M}_h \|w_h\|_{V(h)} \|v_h\|_{V_h}
\quad\text{for all } w_h\in V(h),\ v_h\in V_h.
\]
\end{lemma}
\begin{proof}
The representation of the discrete bilinear form $a_h$ in \eqref{eq2020:kap_85_8631}
can obviously be split into eight sums.
Based on the assumptions, the first three sums,
the sixth sum, and the eighth sum can be estimated by means of the the Cauchy--Schwarz--Bunyakovsky
inequality (using only $\|\cdot\|_{V_h}$).
In order to estimate the fourth and the fifth sums, we first introduce the factors
$(\mu/h_F)^{1/2}(h_F/\mu)^{1/2}$
and only then apply the Cauchy--Schwarz--Bunyakovsky inequality to control the terms
to the expense of the new terms in $\|\cdot\|_{V(h)}$.
For $w_h \in V_h$, the seventh sum can be estimated by means
of a discrete trace inequality as in the proof of Lemma~\ref{LEM:DG_stability}.
If $w_h \in V(h)$, the terms in the second sum are integrated by parts and combined with the seventh sum:
\[
\sum_{K \in\Tnoh} (\nabla \cdot (\vecc w_h),v_h)_K
- \sum_{F \in \overline{\mathcal F}} (\vecc,\jump{w_h v_h})_F
+ \sum_{F \in \mathcal F \cup \mathcal{F}_{2,1}} (\cupw{w_h},\jump{v_h})_F.
\]
Because of the product rule $\nabla \cdot (\vecc w_h) = \nabla \cdot \vecc w_h + \vecc\cdot\nabla w_h$,
the first sum can directly be controlled.
The other two terms either vanish at the boundary faces or can be controlled directly there.
At interior interfaces, we make use of the fact that both terms sum up
to the downwind flux $(\vecc_\textup{down}(w_h),\jump{v_h})_F$
(which is defined analogously to \eqref{eq2020:kap81_824}, \eqref{eq2020:kap81_825}),
and conclude that these terms can be controlled by the third term in \eqref{eq:defDGnorm}.
\end{proof}
To finish the discussion of convergence in the $\|\cdot\|_{V_h}$-norm,
we still need to demonstrate an estimate of the type
\begin{equation}\label{EQ:DG_energy_estimate}
\| u - \Pi u \|_{V(h)} \le C h^m |u|_{m+1,\Tnoh}
\end{equation}
for a suitably chosen element $\Pi u \in V(h)$.
A natural choice for $\Pi$ is the piecewise orthogonal $L^2$-projection.
To verify \eqref{EQ:DG_energy_estimate}, we proceed as follows.
Under the assumption that the family of partitions is shape- and contact-regular,
the left inequality in \eqref{eq:constantsindependent-1},
a multiplicative trace inequality \cite[Lemma 2.19]{Dolejsi:15} and Young's inequality imply
that there exists a constant $C>0$ independent of $h$ such that
\[
\|v\|^2_{V(h)} \le C \sum_{K \in \T_h} \left[ h_K^{-2} \| v
\|^2_{0,K} + | v |^2_{1,K} + h_K^2 | v |^2_{2,K}\right]
\quad\text{for all } v \in V(h).
\]
With the exception of the second term in \eqref{eq:dGnorm}, the estimation of the remaining terms
in \eqref{eq:dGnorm} is largely uncomplicated. We argue as follows:
\begin{align*}
\sum_{F \in \mathcal{F}} \frac{1}{h_F} \| \jump{v} \|^2_{0,F}
& \le C \sum_{K \in\Tnoh} \frac{1}{h_K}\Big[\| \nabla v \|_{0,K}
+ \frac{1}{h_K} \|v\|_{0,K}\Big] \|v\|_{0,K}\\
& \le C \sum_{K \in\Tnoh} \frac{1}{h^2_K} \big[\|v\|^2_{0,K}
+ \|\nabla v \|^2_{0,K}\big],
\end{align*}
where $C>0$ is a generic constant.
The middle term in \eqref{eq:defDGnorm} can be treated analogously by replacing
$\jump{v}$ by $\avg{\vecK^{1/2} \nabla v}$ and $h^{-1}_F$ by $h_F$ in the above estimate.
Now we are prepared to formulate and prove the convergence result in the energy norm.
\begin{theorem}
\label{theo:DG_conv_energy}
Assume that the family of compatible partitions is shape- and contact-regular,
and the coefficients $\vecK, \vecc$ are piecewise continuous.
Furthermore, let the conditions
\eqref{eq:coercivityconddviform}, \eqref{eq:coercivityconddviform2}, \eqref{EQ:stab_condition2},
and one of the conditions \eqref{eq:addcoerccond1} or \eqref{eq:addcoerccond2}
be satisfied.
If $k\in\N$ and the weak solution $u \in V\cap H^{m+1}(\Tnoh)$ with $1 \le m \le k$
of \eqref{eq:-1.33lin}--\eqref{eq2020:kap_32_334_ofdivergence}
satisfies \eqref{eq:7.17a}, then for the IPG solution $u_h\in V_h$ of \eqref{EQ:IPG_global}
the estimate
\[
\|u - u_h\|_{V_h} \le C h^m|u|_{m+1,\Tnoh}
\]
holds with a constant constant $C>0$ independent of $h$.
\end{theorem}
\begin{proof}
Making use of Remark \ref{folg:DG_analysis_conv},~2) with $W\subset H^{m+1}(\Tnoh)$
and Lemmata~\ref{LEM:stability_IPDG}, \ref{LEM:DG_stability},
\ref{LEM:continuity_in_orthogonal_spaces}, it remains to complete the estimate
\[
\inf_{w_h \in V_h} \| u - w_h \|_{V(h)} \le \| u - \Pi u \|_{V(h)}.
\]
This is possible thanks to \eqref{EQ:DG_energy_estimate}.
\end{proof}

\subsubsection*{Convergence order a weaker norms}
Theorem~\ref{theo:DG_conv_energy} trivially implies a (nonoptimal) $L^2$-convergence result.
\begin{theorem}
Let the assumptions of Theorem~\ref{theo:DG_conv_energy} be satisfied.
Then the IPG solution $u_h\in V(h)$ of \eqref{EQ:IPG_global} converges with order at least $m$ to
weak solution $u \in V\cap H^{m+1}(\Tnoh)$ with $1 \le m \le k$
of \eqref{eq:-1.33lin}--\eqref{eq2020:kap_32_334_ofdivergence}
with respect to the $L^2(\Om)$-norm:
\[
\|u - u_h\|_0 \le C h^m |u|_{m+1,\Tnoh}\,.
\]
\end{theorem}

A better result can be obtained by applying Theorem~\ref{l:weakererror}.
To do this we have to study the adjoint problems.
From Lemma~\ref{LEM:stability_IPDG} it is known that, for $k\in\N$,
the original (``primal'') IPG methods \eqref{EQ:IPG_global} are consistent.
If we succeed in showing that the corresponding discrete adjoint problems are also consistent,
then even the special case \eqref{eq2020:kap61_61781} of Theorem~\ref{l:weakererror} can be applied.

It is not difficult to show that, under the same conditions as for the original problem
(see Lemma~\ref{lem2020:kap_85_816'}),
the adjoint problem \eqref{eq:adjoint} with the forms \eqref{eq:adjoint1} possesses a unique solution $v=v_g\in V$.

The discrete adjoint forms read as
\[
\begin{aligned}
a_h^\prime(v_h,w_h) &:= a_h(w_h,v_h)
= \sum_{K \in\Tnoh} \big[(\vecK \nabla w_h - \vecc w_h, \nabla v_h)_K + (rw_h,v_h)_K\big]\\
& + \sum_{F \in \mathcal{F} \cup \mathcal{F}_3}
\Big[\theta(\avg{\vecK \nabla v_h},\jump{w_h})_F
- \Big(\avg{\vecK \nabla w_h}
- \frac{\mu}{h_F}\jump{w_h},\jump{v_h}\Big)_F\Big]\\
& + \sum_{F \in \mathcal{F} \cup \mathcal{F}_{2,1}} (\cupw{w_h},\jump{v_h})_F
+ \sum_{F \in \mathcal{F}_{2,2}} (\talpha w_h, v_h)_F\\
&= \sum_{K \in\Tnoh} \big[(\vecK \nabla v_h, \nabla w_h)_K
- (\vecc\cdot\nabla v_h, w_h)_K + (rv_h,w_h)_K\big]\\
& + \sum_{F \in \mathcal{F} \cup \mathcal{F}_3}
\Big[\theta(\avg{\vecK \nabla v_h},\jump{w_h})_F
- \Big(\jump{v_h},\avg{\vecK \nabla w_h}
- \frac{\mu}{h_F}\jump{w_h}\Big)_F\Big]\\
& + \sum_{F \in \mathcal{F} \cup \mathcal{F}_{2,1}} (\jump{v_h},\cupw{w_h})_F
+ \sum_{F \in \mathcal{F}_{2,2}} (\talpha v_h, w_h)_F\\
&= \sum_{K \in\Tnoh} \big[(\vecK \nabla v_h, \nabla w_h)_K
- (\vecc\cdot\nabla v_h, w_h)_K + (rv_h,w_h)_K\big]\\
& + \sum_{F \in \mathcal{F} \cup \mathcal{F}_3}
\Big[- (\avg{\vecK \nabla w_h,\jump{v_h}})_F
+\theta(\avg{\vecK \nabla v_h},\jump{w_h})_F
+ \frac{\mu}{h_F}(\jump{v_h},\jump{w_h})_F\Big]\\
& + \sum_{F \in \mathcal{F} \cup \mathcal{F}_{2,1}} (\jump{v_h},\cupw{w_h})_F
+ \sum_{F \in \mathcal{F}_{2,2}} (\talpha v_h, w_h)_F,\\
\ell_h(w_h) &:= \tilde\ell(w_h).
\end{aligned}
\]
To investigate the consistency we observe that
\begin{align*}
a_h^\prime(v,w_h) & - \tilde\ell_h(w_h)\\
= & \sum_{K \in\Tnoh} \big[(\vecK \nabla v, \nabla w_h)_K
- (\vecc\cdot\nabla v, w_h)_K + (rv,w_h)_K\big]\\
& + \sum_{F \in \mathcal{F} \cup \mathcal{F}_3}
\Big[- (\avg{\vecK \nabla w_h,\jump{v}})_F
+\theta(\avg{\vecK \nabla v},\jump{w_h})_F
+ \frac{\mu}{h_F}(\jump{v},\jump{w_h})_F\Big]\\
& + \sum_{F \in \mathcal{F} \cup \mathcal{F}_{2,1}} (\jump{v},\cupw{w_h})_F
+ \sum_{F \in \mathcal{F}_{2,2}} (\talpha v, w_h)_F - (g,w_h)\\
= & \sum_{K \in\Tnoh} \big[(\vecK \nabla v, \nabla w_h)_K
- (\vecc\cdot\nabla v, w_h)_K + (rv,w_h)_K\big]\\
& + \sum_{F \in \mathcal{F} \cup \mathcal{F}_3} \theta(\avg{\vecK \nabla v},\jump{w_h})_F\\
& + \sum_{F \in \mathcal{F}_{2,1}} (\jump{v},\cupw{w_h})_F
+ \sum_{F \in \mathcal{F}_{2,2}} (\talpha v, w_h)_F - (g,w_h).
\end{align*}
Here we have used that $\jump{v}_F=0$ on $F \in \mathcal{F} \cup \mathcal{F}_3$.
Next we integrate by parts the first term and obtain
\begin{align*}
a_h^\prime(v,w_h) & - \tilde\ell_h(w_h)\\
= & (- \nabla \cdot \left( \vecK\nabla v \right)- \vecc\cdot\nabla v + rv,w_h)
+\sum_{K \in\Tnoh} (\vecnu \cdot \vecK\nabla v , w_h)_{\partial K}\\
& + \sum_{F \in \mathcal{F} \cup \mathcal{F}_3} \theta(\avg{\vecK \nabla v},\jump{w_h})_F\\
& + \sum_{F \in \mathcal{F}_{2,1}} (\jump{v},\cupw{w_h})_F
+ \sum_{F \in \mathcal{F}_{2,2}} (\talpha v, w_h)_F - (g,w_h)\\
\stackrel{\eqref{eq:magicf}}{=} & (- \nabla \cdot \left( \vecK\nabla v \right) - \vecc\cdot\nabla v + rv,w_h)
+ \sum_{F\in \mathcal{F}}[(\avg{\vecK\nabla v},\jump{w_h})_F + (\jump{\vecK\nabla v},\avg{w_h})_F]\\
& + \sum_{F \in \partial\mathcal{F}} (\vecK\nabla v , \jump{w_h})_F
+ \sum_{F \in \mathcal{F} \cup \mathcal{F}_3} \theta(\avg{\vecK \nabla v},\jump{w_h})_F\\
& + \sum_{F \in \mathcal{F}_{2,1}} (\jump{v},\cupw{w_h})_F
+ \sum_{F \in \mathcal{F}_{2,2}} (\talpha v, w_h)_F - (g,w_h)\\
\stackrel{\eqref{eq:strongadjointde}}{=} &
\sum_{F\in \mathcal{F}}[(\avg{\vecK\nabla v},\jump{w_h})_F + (\jump{\vecK\nabla v},\avg{w_h})_F]\\
& + \sum_{F \in \partial\mathcal{F}} (\vecK\nabla v , \jump{w_h})_F
+ \sum_{F \in \mathcal{F} \cup \mathcal{F}_3} \theta(\avg{\vecK \nabla v},\jump{w_h})_F\\
& + \sum_{F \in \mathcal{F}_{2,1}} (\jump{v},\cupw{w_h})_F
+ \sum_{F \in \mathcal{F}_{2,2}} (\talpha v, w_h)_F.
\end{align*}
The second term in the first sum vanishes due to the regularity assumption
w.r.t.\ the adjoint solution $\vecK\nabla v \in H(\div;\Om)$ (analogously to \eqref{eq:7.17a}).
For the third term, it holds
\[
\sum_{F \in \partial\mathcal{F}} (\vecK\nabla v , \jump{w_h})_F
=\sum_{F \in \mathcal{F}_2 \cup \mathcal{F}_3} (\vecK\nabla v , \jump{w_h})_F
\]
thanks to the homogeneous boundary condition to $\vecK\nabla v \cdot \vecnu$
on $\wtGamma_1$.
Hence
\begin{align*}
a_h^\prime(v,w_h) - \tilde\ell_h(w_h)
= & (1+\theta)\sum_{F\in \mathcal{F} \cup \mathcal{F}_3}(\avg{\vecK\nabla v},\jump{w_h})_F\\
& + \sum_{F \in \mathcal{F}_{2,1}} (\vecK\nabla v , \jump{w_h})_F
+ \sum_{F \in \mathcal{F}_{2,1}} (\jump{v},\cupw{w_h})_F\\
& + \sum_{F \in \mathcal{F}_{2,2}} (\vecK\nabla v , \jump{w_h})_F
+ \sum_{F \in \mathcal{F}_{2,2}} (\talpha v, w_h)_F.
\end{align*}
Since $\vecnu\cdot\vecc = \talpha \ge 0$ on $\wtGamma_{2,1}$ by assumption
\eqref{eq:coercivityconddviform}, 2), all the boundary terms vanish due to the
homogeneous boundary conditions on $\wtGamma_2$, so that we arrive at the
representation
\begin{equation}\label{eq:discreteadjconsistency}
a_h^\prime(v,w_h) - \tilde\ell_h(w_h)
= (1+\theta)\sum_{F\in \mathcal{F} \cup \mathcal{F}_3}(\avg{\vecK\nabla v},\jump{w_h})_F.
\end{equation}
This shows that the adjoint discrete problem is consistent only for $\theta=-1$,
\ie, for the SIPG method.

The arguments from the proof Lemma~\ref{LEM:stability_IPDG} also apply to the adjoint problem in the SIPG case
and provide a consistent method with a unique solution such that a convergence order estimate analogous
to Theorem~\ref{theo:DG_conv_energy} is available.

Hence it is sufficient to estimate the consistency error term \eqref{eq2020:kap61_61781}, that is
\begin{equation}\label{eq:finalconserr}
\begin{aligned}
(a - a_h) (u, v_g)
&= \sum_{F \in \mathcal{F} \cup \mathcal{F}_3}
\Big[(\avg{\vecK \nabla v_g},\jump{u})_F
+ \Big(\avg{\vecK \nabla u}
- \frac{\mu}{h_F}\jump{u},\jump{v_g}\Big)_F\Big]\\
&\quad - \sum_{F \in \mathcal{F}} (\cupw{u},\jump{v_g})_F,
\end{aligned}
\end{equation}
where $u \in V$ and $v_g\in V$ are the weak solutions of \eqref{eq:-1.33lin}--\eqref{eq2020:kap_32_334_ofdivergence}
and \eqref{eq:specialadjoint}, respectively.
Now we can formulate the main result.
\begin{theorem}\label{THEO:conv_SIPG_l2}
Assume that the family of compatible partitions is shape- and contact-regular,
and the coefficients $\vecK, \vecc$ are piecewise continuous.
Furthermore, let the conditions
\eqref{eq:coercivityconddviform}, \eqref{eq:coercivityconddviform2}, \eqref{EQ:stab_condition2},
and one of the conditions \eqref{eq:addcoerccond1} or \eqref{eq:addcoerccond2}
be satisfied.
Let $u \in V\cap H^{k+1}(\Tnoh)$ be the weak solutions of \eqref{eq:-1.33lin}--\eqref{eq2020:kap_32_334_ofdivergence}
and $u_h\in V_h$ the discrete solution of the SIPG method.
Further assume that the solution $v_g$ of the adjoint problem \eqref{eq:specialadjoint}
is stable regular, \ie, for any right-hand side $g \in H^m(\Om)$, $0 \le m \le k-1$,
it belongs to $V\cap H^{m + 2}(\Tnoh)$ and satisfies the estimate
$\| v_g \|_{m +2, \Tnoh} \le C_s \|g \|_{m, \Om}\,$ with some constant $C_s>0$.
Finally, let the solutions $u$ and $v_g$ satisfy \eqref{eq:7.17a}.
Then, there exists a constant $C>0$ independent of $h$ such that
\[
\|u - u_h\|_{-m,\Om} \le C h^{k+m+1} |u|_{k+1}\Tnoh.
\]
\end{theorem}
\begin{proof}
From \eqref{eq:finalconserr} it can be seen that the regularity assumptions
together with the boundary conditions yield immediately
\[
(a-a_h)(u,v_g)=0.
\]
\end{proof}
\begin{remark}
According to \eqref{eq:discreteadjconsistency},
for other methods with $\theta\ne -1$,
the discretization of the adjoint problem is no longer consistent to the adjoint problem,
\ie, according to Theorem~\ref{l:weakererror}, the second term
The relationship \eqref{eq:discreteadjconsistency} indicates that the discretization of the adjoint problem
is no longer consistent to the adjoint problem if $\theta\ne -1$.
Then, according to Theorem~\ref{l:weakererror}, the second term in the bound \eqref{eq2020:kap61_6174},
that is
\[
a_h(u-u_h,v_g) - \tilde\ell_h(u-u_h)
= (1+\theta)\sum_{F\in \mathcal{F} \cup \mathcal{F}_3}(\avg{\vecK\nabla v_g},\jump{u-u_h})_F,
\]
has still be estimated appropriately.
\end{remark}

\section{Conclusion}
We presented a unified approach to the analysis of FEM for boundary value problems
with linear diffusion-convection-reaction equations and  boundary conditions of mixed type,
where neither conformity nor consistency properties are assumed.
Being elementary in nature, it clarifies and quantifies the interplay between
stability, approximation errors, and consistency errors
-- the theme guiding PDE numerical analysis from its very beginning.
As an example, we formulated and investigated two different stabilized discretizations
and obtained stability and optimal error estimates in energy-type norms and,
as a consequence of our generalization of the Aubin-Nitsche technique,
optimal error estimates in weaker norms.
We expect the described framework to provide guidelines to set up and analyze
further new stable and convergent schemes.

\end{document}